\documentclass[11pt, oneside]{amsart}
\usepackage{wrapfig}
\usepackage{tikz}
\usepackage{comment}
\usepackage{mathrsfs}
\usepackage{tabularx, hyperref}
\usepackage{amssymb} \usepackage{amsfonts} \usepackage{amsmath}
\usepackage{amsthm} \usepackage{epsfig, subfig}
\usepackage{ amscd, amsxtra, latexsym}
\usepackage{epsfig,  graphicx, psfrag}
\usepackage[all]{xy}
\usepackage{caption}
\usepackage{enumerate}
\usepackage{color}

\DeclareMathOperator\ann{Ann}

\newtheorem{lemma}{Lemma}[section]
\newtheorem{thm}[lemma]{Theorem}
\newtheorem{prop}[lemma]{Proposition}
\newtheorem{cor}[lemma]{Corollary}

\newtheorem*{prop*}{Proposition}
\newtheorem{prop_intro}{Proposition}
\newtheorem{conj_intro}[prop_intro]{Conjecture}
\newtheorem{quest_intro}[prop_intro]{Question}
\newtheorem{thm_intro}[prop_intro]{Theorem}

\newtheorem{cor_intro}[prop_intro]{Corollary}
\theoremstyle{definition}
\newtheorem{defn_intro}[prop_intro]{Definition}
\newtheorem{defn}[lemma]{Definition}

\newtheorem{rem}[lemma]{Remark}

\theoremstyle{definition}

\newtheoremstyle{citing}
  {3pt}
  {3pt}
  {\itshape}
  {}
  {\bfseries}
  {}
  {.5em}
  {\thmnote{#3}}
\theoremstyle{citing}
\newtheorem*{varthm}{}

\definecolor{darkgreen}{cmyk}{1,0,1,.2}

\newcommand{\ha}{\ensuremath {H_2^{\rm am}}}

\newcommand{\N}{\ensuremath {\mathbb{N}}}
\newcommand{\R} {\ensuremath {\mathbb{R}}}

\newcommand{\Z} {\ensuremath {\mathbb{Z}}}
\newcommand{\G} {\ensuremath {\Gamma}}

\newcommand{\calH} {\ensuremath {\mathcal{H}}}

\renewcommand{\phi}{\varphi}

\begin{document}

\title[]{Central extensions and bounded cohomology}

\author[]{Roberto Frigerio}
\address{Dipartimento di Matematica, Universit\`a di Pisa, Italy}
\email{roberto.frigerio@unipi.it}

\author[Alessandro Sisto]{Alessandro Sisto}
	\address{Department of Mathematics, Heriot-Watt University, Edinburgh, UK}
	\email{a.sisto@hw.ac.uk}


\date{\today}

\begin{abstract}
It was shown by Gersten that a central extension of a finitely generated group is quasi-isometrically trivial provided that its Euler class is bounded. 
We say that a finitely generated group $G$ satisfies Property QITB (quasi-isometrically trivial implies bounded) if the Euler class of any quasi-isometrically trivial
central extension of $G$ is bounded. We exhibit a finitely generated group $G$ which does not satisfy Property QITB. 
This answers a question by Neumann and Reeves, and provides partial answers to related questions by Wienhard and Blank. We also prove that Property QITB holds
for a large class of groups, including amenable groups,  right-angled Artin groups,  relatively hyperbolic groups with amenable peripheral subgroups, and 3-manifold groups.

Finally, we show that 
Property QITB holds for every finitely presented group if a conjecture by Gromov on bounded primitives of differential forms holds as well.
\end{abstract}

\maketitle

\section{Introduction}
Let
$$
\xymatrix{
1\ar[r] & Z\ar[r]^{i} & E \ar[r]^{\pi}&  G \ar[r] & 1 
}
$$
be a central extension of groups, where $G$ and $Z$ (hence, $E$) are finitely generated. Any such extension defines a cohomology class
$\omega\in H^2(G,Z)$, which will be called the \emph{Euler class} of the extension. It is well known that 
the Euler class completely determines the isomorphism class of a central extension (see Section~\ref{prelim:sec}), 
and it is natural to investigate which geometric features it encodes.

We say that a class $\omega\in H^2(G,Z)$
is \emph{bounded} if it lies in the image of the comparison map $H^2_b(G,Z)\to H^2(G,Z)$, i.e.~if it can be described by a bounded cocycle
(see Section~\ref{prelim:sec} for the precise definition).
Following~\cite{gersten:preprint, Ger:last, klelee}, we say that the extension of finitely generated groups
$$
\xymatrix{
1\ar[r] & Z\ar[r]^{i} & E \ar[r]^{\pi}&  G \ar[r] & 1 
}
$$
 is \emph{quasi-isometrically trivial} if there exists a quasi-isometry
$$
f\colon E\longrightarrow {Z}\times G
$$
such that the following diagram commutes, up to bounded error:
$$
\xymatrix{ E\ar[r]^{\pi}\ar[d]^{f} &G\ar[d]^{\rm Id}\\
Z\times G\ar[r]^{\pi_2} & G\ .}
$$
Here $\pi_2\colon Z\times G\to G$ is the projection on the second factor, and, for clarity, the diagram commuting up to bounded error means that given a word metric $d_G$ on $G$ we have $\sup_{h\in E} d_G(\pi(h),\pi_2(f(h)))<+\infty$.

It was first shown by Gersten~\cite{Ger:last, gersten:preprint} that a central extension is quasi-isometrically trivial provided that its Euler class
is bounded. In this paper we address the following:

\begin{quest_intro}\label{fund:quest}
Is the Euler class of a quasi-isometrically trivial extension necessarily bounded?
\end{quest_intro}

Question~\ref{fund:quest} was first asked by Neumann and Reeves in~\cite{NeuRee0, NeuRee}
(see also \cite[Remark 2.6]{whyte}).
Moreover, it turns out to be equivalent to questions on $\ell^\infty$-cohomology posed in~\cite{Wie, blank:thesis} (see Question~\ref{linftyquest} and Proposition~\ref{equivalenceform}),
and related to a Conjecture by Gromov (see Conjecture~\ref{Gromov_conj} and Corollary~\ref{conj_cor}).
 We provide here a negative answer to Question~\ref{fund:quest}:

\begin{thm_intro}\label{counterexample:thm}
There exists a quasi-isometrically trivial central extension of a finitely generated group $G$ by $\Z$ whose Euler class is not bounded.\footnote{Since the first version of this paper appeared on ArXiv, a finitely presented example of a group with the same property has been found \cite{AM:no_Gromov}.}
\end{thm_intro}

\begin{defn_intro}
Let $G$ be a finitely generated group. Then $G$ satisfies Property QITB (``quasi-isometrically trivial $\Rightarrow$ bounded'') if the following condition holds:
for every finitely generated abelian group $Z$, the Euler class of any  quasi-isometrically trivial central extension of $G$ by $Z$ is bounded.
\end{defn_intro}

Theorem~\ref{counterexample:thm} states that there exists a finitely generated group $G$ which does not satisfy Property QITB.
Nevertheless, we show that Property QITB holds for large families of groups:

\begin{thm_intro}\label{families}
Suppose the finitely generated group $G$ belongs to one of the following families:
\begin{enumerate}
\item amenable groups;
\item relatively hyperbolic groups with respect to a finite family of amenable peripheral subgroups (in particular, hyperbolic groups);
\item right-angled Artin groups;
\item fundamental groups of compact orientable
$3$-manifolds.
\end{enumerate}
Then $G$ satisfies Property QITB.
\end{thm_intro}

For amenable and right-angled Artin groups we can prove more, at least when we extend by a torsion-free abelian group (meaning that we consider a central extension where the group on the left is torsion-free abelian). We call a central extension of $G$ \emph{virtually trivial} if it pulls back to a trivial central extension on a finite-index subgroup of $G$, see Definition \ref{defn:virtually_trivial}. It is easy to check that virtually trivial central extensions are quasi-isometrically trivial. 

\begin{thm_intro}\label{amenable:vt}
Let $G$ be a finitely generated amenable group, or a finitely generated right-angled Artin group. Then a central extension of $G$ by a  finitely generated 
torsion-free abelian group is quasi-isometrically
trivial if and only if it is virtually trivial. 
\end{thm_intro}

We remark that the torsion-freeness assumption cannot be dropped, see Remark \ref{rem:Thompson}. However, even without that assumption, one can replace ``if and only if it is virtually trivial'' with ``if and only if its Euler class has finite order'', see Theorem \ref{*virt}.

Our strategy to prove Theorems~\ref{families} and~\ref{amenable:vt} is as follows. We first show that they hold for  amenable groups. Then, for every group $G$, we 
introduce  the subspace $\ha(G,\R)\subseteq H_2(G,\R)$ 
generated by those elements of $H_2(G,\R)$ which lie in the image of a map $f_*\colon H_2(A,\R)\to H_2(G,\R)$, where
$f\colon A\to G$ is a homomorphism and $A$ is amenable. 
Whenever $\ha(G,\R)=H_2(G,\R)$ we are able to prove that any quasi-isometrically trivial extension by a torsion-free abelian group is virtually trivial. In order to prove
 Property QITB for all the groups listed in the statement of Theorem~\ref{families}, we then observe that such property holds whenever every element in the annihilator
of $\ha(G,\R)$ is bounded. 

\smallskip

Further examples of groups satisfying Property QITB may be built thanks to the following results:

\begin{prop_intro}\label{direct:prop}
Let $G_1$, $G_2$ be groups satisfying Property QITB. Then
the direct product $G_1\times G_2$ satisfies Property QITB.
\end{prop_intro}


\begin{prop_intro}\label{amalgamated:prop}
Let $G=G_1*_H G_2$ be a \emph{transverse} amalgamated product, where $H$ is amenable. If $G_1,G_2$ satisfy Property QITB, then
$G$ satisfies Property QITB.
\end{prop_intro}

We refer the reader to Definition~\ref{transverse:defn} for the notion of transverse amalgamated product. Free products are particular cases of 
transverse amalgamated products, hence we get the following:

\begin{cor_intro}\label{free:cor}
If $G_1,G_2$ satisfy Property QITB, then the free product $G_1*G_2$ satisfies QITB.
\end{cor_intro}

Using Propositions \ref{direct:prop} and \ref{amalgamated:prop} one can prove that the fundamental groups of many higher dimensional graph manifolds defined in \cite{FLS:graph} satisfy QITB.

Before investigating the relationship between Property QITB and other cohomological properties of groups, we ask here the following question,
which shows that,
surprisingly enough, the geometry of central extensions seems to still be quite elusive.

\begin{quest_intro}\label{undistorted}
If $1 \to Z\to E\to G\to 1$ is a quasi-isometrically trivial extension, then $Z$ is undistorted in $E$. There is no apparent reason why the converse
of this statement should also hold. Therefore, we ask here the following question:
does there exist a non-quasi-isometrically trivial extension $1 \to Z\to E\to G\to 1$ for which $Z$ is undistorted in $E$?
\end{quest_intro}

\subsection*{Weakly bounded cochains}
Let $A$ be an abelian group
(we will be mainly interested in the cases when either $A$ is finitely generated, or $A=\R$). 
We denote by $C^*(G,A)$ the bar resolution of $G$ with coefficients in $A$ (see e.g.~\cite[Chapter III]{brown}).  Recall that a cochain 
$\omega\in C^n(G,A)$ (which is a map $\omega\colon G^n\to A$) is \emph{bounded} if the set
$
\omega(G^n)
$
is bounded as a subset of $A$ (if $A$ is a finitely generated abelian group, this amounts to asking that $\omega(G^n)$ be finite). 
Following~\cite{NeuRee0,NeuRee} (where only the case of degree 2 was considered),
we say that a cochain $\omega \in C^n(G,A)$ is \emph{weakly bounded}
 if, for every fixed $(n-1)$-tuple
$(g_2,\ldots,g_n)\in G^{n-1}$, the set
$$
\omega(G,g_2,\ldots,g_n)\ \subseteq \ A
$$
is bounded.

Bounded cochains provide a subcomplex $C^*_b(G,A)$ of $C^*(G,A)$ (while weakly bounded cochains do not). The cohomology of $C^*_b(G,A)$
is denoted by $H^*_b(G,A)$. The inclusion of bounded cochains into ordinary cochains induces the \emph{comparison map}
$$
c^*\colon H^*_b(G,A)\to H^*(G,A)\ .
$$
We say that a class $\alpha\in H^*(G,A)$ is \emph{bounded} if it may be represented by a bounded cocycle, i.e.~if it lies in the image of the comparison map $c^*$,
and \emph{weakly bounded} if it may be represented by a weakly bounded cocycle.

Suppose now that $A$ is finitely generated.
As mentioned above, Gersten proved that a central extension is quasi-isometrically trivial provided it may be described by a bounded cocycle. 
Neumann and Reeves then observed that a central extension of a finitely generated group is quasi-isometrically trivial if and only if its Euler class is  weakly bounded
(see Corollary~\ref{weak:fund:cor}). Therefore, Theorem~\ref{counterexample:thm} implies the following:

\begin{cor_intro}\label{neumann:cor}
There exist a finitely generated group $G$ and a class $\alpha\in H^2(G,\mathbb{Z})$ such that $\alpha$ is   weakly bounded, but not bounded.
\end{cor_intro}

Moreover,  Theorem~\ref{families} and Propositions~\ref{direct:prop} and~\ref{amalgamated:prop} imply that   weak boundedness and boundedness are
indeed equivalent (in degree 2) for a large class of groups.

 \subsection*{$\ell^\infty$-cohomology}
 A key ingredient in our proof of Theorem~\ref{counterexample:thm} is the fact that
 weakly bounded classes may be characterized in terms of the  $\ell^\infty$--cohomology $H^*_{(\infty)}(G,A)$ of $G$, which was first defined
by Gersten in~\cite{gersten:preprint,Ger:last}, and further studied e.g.~in~\cite{gersten,mineyevell1, mineyevell2,BNW,blank:thesis,Milizia}. 

The $\ell^\infty$--cohomology of a group $G$ in degree $\leq n$
was originally defined via the cellular cohomology complex of an Eilenberg-MacLane space $X$ for $G$, under the assumption that $X$ has a finite $n$-skeleton.
It was observed by Wienhard~\cite[Section 5]{Wie} (see also~\cite[Section 6.3]{blank:thesis}) that $\ell^\infty$--cohomology may be defined in purely algebraic terms 
(i.e.~without referring to any cellular complex providing a model for $G$, hence without restricting to groups admitting models with finite skeleta). 
Recently, Milizia~\cite{Milizia} extended Gersten's definition of  the $\ell^\infty$-cohomology 
$H^*_{(\infty)}(X,A)$
of a cellular complex $X$ to avoid any assumption on the skeleta of $X$. He also proved that 
 $H^*_{(\infty)}(X,A)$ is canonically isomorphic to $H^*_{(\infty)}(G,A)$ whenever $X$ is an Eilenberg-MacLane space for $G$. This approach has the advantage to provide a topological interpretation of $H^*_{(\infty)}(G,A)$
also when $G$ is not finitely presented (or even finitely generated). This topological interpretation will prove very useful in our proof of Theorem~\ref{counterexample:thm}.

The $\ell^\infty$-cohomology
of a group comes with a natural
map $\iota^*\colon H^*(G,A)\to H^*_{(\infty)}(G,A)$. In Section~\ref{linfty:sec} we provide a direct proof of the following:

\begin{prop_intro}\label{fund:prop}
Let $\alpha\in H^n(G,A)$, $n\in\mathbb{N}$. Then $\alpha$ is  weakly bounded if and only if $\iota^n(\alpha)=0$.
\end{prop_intro}

As a corollary, we recover the following characterization of quasi-isometrically trivial central extensions, which was proved by Kleiner and Leeb
via a different strategy (see also~\cite[Theorem 0.3]{whyte}):

\begin{thm_intro}[{\cite[Theorem 1.8]{klelee}}]\label{klelee:thm}
Let $\alpha\in H^2(G,Z)$ be the Euler class of a central extension of a finitely generated group. Then $\iota^2(\alpha)=0$ in $H^2_{(\infty)}(G,Z)$ if and only if
the extension is quasi-isometrically trivial.
\end{thm_intro}

 Since bounded classes are weakly bounded, Proposition~\ref{fund:prop} implies that 
the composition
 \begin{equation}\label{exact:eq}
 \xymatrix{
 H^n_b(G,A) \ar[r]^{c^n} & H^n(G,A) \ar[r]^{\iota^n} & H^n_{(\infty)}(G,A)
 }
 \end{equation}
is the zero map for every $n\in\mathbb{N}$, a result which was already known (at least for $A=\R$) to Gersten (see~\cite[Proposition 10.3]{gersten} for the case where
$G$ admits an Eilenberg-MacLane model
with a finite $n$-skeleton, and~\cite{Wie} or \cite{blank:thesis} for the general case).
The following question was posed (for $A=\mathbb{R}$) by Wienhard in~\cite[Question 8]{Wie} and by Blank in~\cite[Question 6.3.10]{blank:thesis}:

\begin{quest_intro}\label{linftyquest}
When is the sequence~\eqref{exact:eq} exact?
\end{quest_intro}

The results proved in this paper partially answer this question. Indeed,
since the kernel of $\iota^n$ coincides with the space of weakly bounded classes,
 the sequence~\eqref{exact:eq} is exact if and only if 
every weakly bounded $n$-class  is bounded. For a fixed group $G$, it is not difficult to show that this condition
holds for $A=\mathbb{Z}$ if and only if it holds for every $A$, as 
$A$ varies in the class of finitely generated abelian groups
  (see Proposition~\ref{ZntoZ}). Therefore, we have the following:

\begin{prop_intro}\label{equivalenceform}
The sequence
$$\xymatrix{
 H^2_b(G,\mathbb{Z}) \ar[r]^{c^n} & H^2(G,\mathbb{Z}) \ar[r]^{\iota^n} & H^2_{(\infty)}(G,\mathbb{Z})
 }$$
is exact if and only if
the group $G$ satisfies Property QITB.
\end{prop_intro}

Some subtleties arise when comparing the case with integral coefficients with the case with real coefficients.
By Lemma~\ref{ZtoR1}  and Proposition~\ref{weakly:linf}, if the sequence
\begin{equation}\label{exactR:eq}
 \xymatrix{
 H^n_b(G,\R) \ar[r]^{c^n} & H^n(G,\R) \ar[r]^{\iota^n} & H^n_{(\infty)}(G,\R)
 }
 \end{equation}
 is exact, then the sequence
\begin{equation}\label{exactZ:eq}
 \xymatrix{
 H^n_b(G,\Z) \ar[r]^{c^n} & H^n(G,\Z) \ar[r]^{\iota^n} & H^n_{(\infty)}(G,\Z)
 }
 \end{equation}
 is also exact.  As a corollary of Theorem~\ref{counterexample:thm} we then have the following:

\begin{cor_intro}
Let $n=2$.
There exists a finitely generated group $G$ for which the sequences~\eqref{exactR:eq} and~\eqref{exactZ:eq}
 are \emph{not} exact. 
\end{cor_intro}

On the contrary, the same sequences are exact for all the groups listed in the statement of Theorem~\ref{families} (see Section~\ref{families:sec}).

  The question whether 
  the exactness of sequence~\eqref{exactZ:eq} implies the exactness of~\eqref{exactR:eq} 
 seems quite tricky: it would hold, for example, 
provided that the answer to Question~\ref{coefficients:quest}--(2) were positive.
  
 \begin{quest_intro}\label{coefficients:quest}
Let us say that a class $\alpha\in H^{n}(G,\R)$ is \emph{integral} if it belongs to the image of $H^n(G,\Z)$ via the change of coefficients homomorphism.
 As a consequence of Lemma~\ref{ZtoR1}, an integral class $\alpha\in H^{n}(G,\R)$ is (weakly) bounded if and only if it is the image of a (weakly) bounded class  
 in $ H^n(G,\Z)$ under the change of coefficients homomorphism.
  \begin{enumerate}
 \item Does every bounded class in $H^n(G,\R)$ belong to the $\R$-linear subspace spanned by integral bounded classes? Or, equivalently, does the image of
 $H^n_b(G,\Z)$ in $H^n(G,\R)$ span (over $\R$)  the space of bounded classes in  $H^n(G,\R)$?
 \item Does every weakly bounded class in $H^n(G,\R)$ belong to the $\R$-linear subspace spanned by weakly bounded integral classes? 
 \end{enumerate}
 \end{quest_intro}

Our example of a finitely generated group which does not satisfy Property QITB is not finitely presented.
In Section~\ref{linfty:sec} we show that in higher degrees it is possible (and much easier) to find finitely presented groups
supporting weakly bounded classes that are not bounded:

\begin{prop_intro}\label{many:counterexamples}
For every $n\geq 3$, there exists a finitely presented group for which the sequences~\eqref{exactR:eq} and~\eqref{exactZ:eq}  are not exact.
\end{prop_intro}

As discussed in the following subsection, the situation in degree 2 seems to be very different.

\subsection*{Property QITB and bounded differential forms}
 In~\cite{gromovasymptotic},
Gromov proposed the following:

\begin{conj_intro}[{\cite[page 93]{gromovasymptotic}}]\label{Gromov_conj}
Let $V$ be a closed Riemannian manifold, and let $\alpha\in H^2(V,\R)$.
Then $\alpha$ is $\widetilde{d}$-bounded if and only if it is bounded.\footnote{In view of the finitely presented example mentioned in the previous footnote, and Corollary \ref{conj_cor}, this conjecture is now known to be false.}
\end{conj_intro}

Recall from~\cite{gromovkah,gromovasymptotic} that a class $\alpha\in H^2(V,\R)$ is $\widetilde{d}$-bounded if the following holds:
if $\omega\in\Omega^2(V)$ is a closed differential form representing $\alpha$ via the de Rham isomorphism, and
$\widetilde{\omega}\in \Omega^2(\widetilde{V})$ is the lift of $\omega$ to the universal covering $\widetilde{V}$ of $V$, 
then $\widetilde{\omega}=d\varphi$ for some $\varphi\in \Omega^1(\widetilde{V})$ such that
 $\sup_{x\in\widetilde{V}} |\varphi_x|<+\infty$.
Moreover, $\alpha$ is bounded if it lies in the image of the comparison map between the singular bounded cohomology of $V$ and the usual singular cohomology of $V$,
i.e.~if it admits a representative $c$ in the singular chain complex such that
$c(\sigma)$ is uniformly bounded as $\sigma$ varies among all the singular simplices in $V$.

The study of the growth of primitives in non-compact manifolds was initiated by Sullivan~\cite{Sullivan2}, Gromov~\cite{grom:hypmani} and Brooks~\cite{brooks-lap}
and has then been proved to be closely related to coarse invariants of (fundamental) groups (see e.g.~\cite{Zuk,Nowak}). 
We refer the reader e.g.~\cite{sikorav} for a brief account on the topic, and for a self-contained proof of the fact that bounded classes are $\widetilde{d}$-bounded, and
to~\cite{CHI,BIgeo, Wie} for further developments of the theory.

In Section~\ref{linfty:sec} we prove the following:

\begin{thm_intro}\label{equiv:thm}
Let $V$ be a closed Riemannian manifold. Then $V$ satisfies the statement of Conjecture~\ref{Gromov_conj} if and only if
the sequence 
$$
\xymatrix{
 H^2_b(\pi_1(V),\R) \ar[r]^{c^n} & H^2(\pi_1(V),\R) \ar[r]^{\iota^n} & H^2_{(\infty)}(\pi_1(V),\R)
 }
$$
is exact (and, if this condition holds, then $\pi_1(V)$ satisfies QITB).
\end{thm_intro}

Since the class of fundamental groups of compact Riemannian manifolds coincides with the class of finitely presented groups,
we obtain the following:

\begin{cor_intro}\label{conj_cor}
If Conjecture~\ref{Gromov_conj} holds and $G$ is a finitely presented group, then 
the sequence
$$
\xymatrix{
 H^2_b(G,\R) \ar[r]^{c^2} & H^2(G,\R) \ar[r]^{\iota^2} & H^2_{(\infty)}(G,\R)
 }
 $$
 is exact. In particular, Gromov's Conjecture would imply that every finitely presented group satisfies QITB. 
In fact, Gromov's conjecture would be equivalent to the fact that every finitely presented group satisfies QITB provided that the answer to
Question~\ref{coefficients:quest}--(2) were affirmative for every finitely presented group. 
\end{cor_intro}

Note however that Gromov himself stated in~\cite{gromovasymptotic} that ``the evidence in favour of the conjecture is rather limited and it would  be safe
to make some extra assumption on $\pi_1(V)$''.

\subsection*{Plan of the paper} In Section~\ref{prelim:sec} we introduce bounded and weakly bounded cochains, and we prove that a central extension is quasi-isometrically
trivial if and only if its Euler class is weakly bounded. 
 In Section~\ref{linfty:sec} we introduce Gersten's $\ell^\infty$-cohomology, we prove the characterization of weakly bounded cochains described in Proposition~\ref{fund:prop} (which allows us to recover Theorem~\ref{klelee:thm} by Kleiner and Leeb), and we prove Proposition~\ref{many:counterexamples} and Theorem~\ref{equiv:thm}.
Section~\ref{counter:sec} is devoted to the proof of Theorem~\ref{counterexample:thm}, i.e.~to the construction of a finitely
generated group admitting a quasi-isometrically trivial extension with an unbounded Euler class. In Section~\ref{families:sec} we 
construct examples of groups satisfying Property QITB, and we 
prove Theorems~\ref{families} and~\ref{amenable:vt}, and Propositions~\ref{direct:prop} and~\ref{amalgamated:prop}.

\subsection*{Acknowledgements} The authors thank Marc Burger, Francesco Fournier Facio, Konstantin Golubev, Alessandra Iozzi, Andrea Maffei and Francesco Milizia for useful conversations.

A. S. was partially supported by the Swiss National Science Foundation (grant \#182186).

\section{Preliminaries}\label{prelim:sec}
\subsection*{Quasi-isometries}
Let us briefly recall the definition of quasi-isometry. If $(X,d)$, $(Y,d')$ are metric spaces, a map
$f\colon X\to Y$ is a quasi-isometric embedding if there exist constants $k\geq 1$, $c\geq 0$ such that
$$
\frac{d(x_1,x_2)}{k}-c\leq  d(f(x_1),f(x_2))\leq k\cdot d(x_1,x_2)+c\
$$
for every $x_1,x_2\in X$.
A map $g\colon Y\to X$ is a \emph{quasi-inverse} of $f$ if it is a quasi-isometric embedding, and the maps $g\circ f\colon X\to X$, $f\circ g\colon Y\to Y$ are uniformly close to the 
identity of $X$ and $Y$, respectively. A quasi-isometry is a quasi-isometric embedding that admits a quasi-inverse.

If $G$ is a finitely generated group and $S$ is a finite symmetric generating set for $G$ (symmetric meaning that $s\in S$ if and only if $s^{-1}\in S$), then the \emph{Cayley} graph $C_S(G)$ of $G$ with respect to $S$
is the graph having $G$ as set of vertices and $G\times S$ as set of edges, where the edge $(g,s)$ joins $g$ with $gs$. The graph $C_S(G)$ is endowed
with a path metric for which every edge is isometric to a segment of unitary length. It is well known that, if $S,S'$ are finite generating sets for $G$, then the identity
of $G$ extends to a quasi-isometry between $C_S(G)$ and $C_{S'}(G)$. Thus, one can define the quasi-isometry type of $G$ as the quasi-isometry type of any of its Cayley graphs.

\subsection*{Notation} If $g\in G$ then we denote by $\|g\|_S$ the distance in $C_S(G)$ between $g$ and the identity of the group
(i.e.~the minimal number of factors needed to describe $g$ as a product of elements of $S$ and their inverses).



\subsection*{(Weakly) bounded classes}
Let $A$ be an abelian group
(as in the introduction, we assume that either $A$ is finitely generated, or $A=\R$). 
Recall from the introduction that a class $\alpha\in H^*(G,A)$ is \emph{bounded} if it may be represented by a bounded cocycle,
and \emph{ weakly bounded} if it may be represented by a   weakly bounded cocycle. 

A $2$-cocycle $\omega\in C^2(G,A)$ is \emph{normalized} if $\omega(1,G)=\omega(G,1)=0$, where $1$ is the identity of $G$. It is well known that every cohomology
class may be represented by a normalized cocycle.

\begin{lemma}\label{linear:growth}
Let $S=\{x_1,\ldots,x_n\}$ be a symmetric set of generators of $G$, and let 
$\omega\in C^2(G,A)$ be a normalized cocycle. 
 Suppose that $|\omega(g,x_i)|\leq C$ for every $g\in G$, $i=1,\ldots,n$. Then
 $$| \omega(G,g)|\leq  2C \|g\|_S \quad \text{ for every} \ g\in G\ .
 $$
 In particular, $\omega$ is weakly bounded. 
\end{lemma}
\begin{proof}
Let $g$ be an element of $G$. 
We will prove by induction on $\|g\|_S$ that
$$
|\omega(G,g)|\leq 2C\|g\|_S\  .
$$
The case $\|g\|_S=0$ follows from the fact that $\omega$ is normalized.
Assuming the inequality for all elements of length $j-1$, if $\|g\|=j$  and $g=g'x_i$ with $\|g'\|_S=j-1$,
then for every $h\in G$ we have (by using the cocycle relation):
\begin{align*}
|\omega(h,g)| &= |\omega(h,g')-\omega(g',x_i)+\omega(hg',x_i)|
\\ & \leq |\omega(h,g')|+|\omega(g',x_i)|+|\omega(hg',x_i)|\\ &\leq 2C\|g'\|_S+C+C=2C\|g\|_S\ .
\end{align*}
\end{proof}

 \begin{rem}
 Our terminology slightly differs from Neumann and Reeves'. In fact, our weakly bounded cocycles
 correspond to \emph{right} weakly bounded cocycles in the terminology introduced 
in~\cite{NeuRee0, NeuRee}, where a cocycle $\omega\in Z^2(G,A)$ is called 
\emph{left weakly bounded} if $\omega(g,G)$ is a bounded subset of $A$ for every $g\in G$, and 
\emph{weakly bounded} if it is both right weakly bounded and left weakly bounded.
It is not difficult to show that a class in $H^2(G,A)$ admits a left weakly bounded representative if and only if it admits a right weakly bounded representative. 
In fact, it turns out that left weak boundedness, right weak boundedness and weak boundedness are equivalent 
for elements of $H^2(G,A)$, by~\cite[Theorem 4.1]{NeuRee}. 
 \end{rem}

\subsection*{The Euler class of a central extension}
Let us now consider a central extension
$$
\xymatrix{
1\ar[r] & Z\ar[r]^{i} & E \ar[r]^{\pi}&  G \ar[r] & 1 \ ,
}
$$
and let $s\colon G\to E$ be a section of $\pi\colon E\to G$. 
For $g_1,g_2\in G$, the element $s(g_1)s(g_2)s(g_1g_2)^{-1}$ lies in the kernel of $\pi$, hence in the image of $i$. Up to identifying $i(Z)$ with $Z$
we may thus define the cochain $\omega\in C^2(G,Z)$ given by
$$
\omega_s(g_1,g_2)=s(g_1)s(g_2)s(g_1g_2)^{-1}\ \in Z\ .
$$
Let us recall the  following well-known  facts (see e.g.~\cite[Chapter 4]{brown}):
\begin{enumerate}
\item The cochain $\omega_s$ is  a cocycle;
\item If $s'$ is another section of $\pi$, then $\omega_{s'}$ is cobordant to $\omega_s$; therefore, the class $[\omega_s]\in H^2(G,Z)$ does not depend
on the choice of $s$, and will be called the \emph{Euler class} of the extension;
\item If $\omega'$ is any representative of the Euler class, then there exists a section $s'\colon G\to E$ such that 
$\omega'=\omega_{s'}$;
\item The cocycle $\omega_s$ is normalized if and only if $s(1)=1$ (in this case, we say that $s$ is normalized too).
\end{enumerate}

Two central extensions are  isomorphic if they are described by the rows of a commutative diagram as follows:
$$
\xymatrix{
1\ar[r] & Z\ar[r]^{i} \ar[d]^{\rm Id} & E \ar[r]^{\pi}\ar[d]^h&  G \ar[r]\ar[d]^{\rm Id} & 1\\
1\ar[r] & Z\ar[r]^{i} & E' \ar[r]^{\pi}&  G \ar[r] & 1\ ,
}
$$
where $h$ is a homomorphism (hence, an isomorphism). It is readily seen that isomorphic central extensions share the same Euler class. In fact,
it readily follows from the facts listed above that, via  the Euler class, the module $H^2(G,Z)$ classifies central extensions of $G$ by $Z$ up to 
isomorphism.

\subsection*{The Euler class of a quasi-isometrically trivial central extension}
The following characterization of quasi-isometrically trivial extensions 
was proved by Kleiner and Leeb~\cite{klelee} in the case when $Z$ is a finitely-generated torsion free abelian group (i.e.~$Z\cong \mathbb{Z}^n$). 
However, the proof in~\cite{klelee} works verbatim in the general case. 

\begin{prop}[{\cite[Proposition 8.3]{klelee}}]\label{klelee}
Let
$$
\xymatrix{
1\ar[r] & Z\ar[r]^{i} & E \ar[r]^{\pi}&  G \ar[r] & 1 
}
$$
be a central extension. Then the following conditions are equivalent:
\begin{enumerate}
\item The extension is quasi-isometrically trivial.
\item The projection $\pi$ admits a Lipschitz section $s\colon G\to E$.
\end{enumerate}
\end{prop}

It is almost tautological that the $2$-cocycle $\omega$ associated to a section $s\colon G\to E$ is bounded if and only if
$s$ is a \emph{quasihomomorphism} in the sense of Kapovich and Fujiwara~\cite{FuKa}. Therefore, a group $G$ satisfies Property QITB if and only
if the existence of a Lipschitz section for a central extension of $G$ implies the existence of a quasihomomorphic section for the same extension.
We refer the reader to~\cite{heuer} for a discussion of (not necessarily central) extensions with bounded Euler class in terms of quasihomomorphisms.

The following Lemma~\ref{weak:fund} and its immediate Corollary~\ref{weak:fund:cor} 
play a fundamental role in our study of quasi-isometrically trivial central extensions. They are stated 
in~\cite[Section 4]{NeuRee0}.  For the sake of completeness, we provide here a proof of Lemma~\ref{weak:fund}.

\begin{lemma}\label{weak:fund}
Let $s\colon G\to E$ be a normalized section for the central extension
$$
\xymatrix{
1\ar[r] & Z\ar[r]^{i} & E \ar[r]^{\pi}&  G \ar[r] & 1 \ ,
}
$$
and let $\omega\in C^2(G,Z)$ be the associated 2-cocycle. Then $s$ is Lipschitz if and only if $\omega$ is  weakly bounded.
\end{lemma}
\begin{proof}
Suppose $s$ is Lipschitz. 
Let $S=\{x_1,\ldots,x_n\}$ be a symmetric set of generators for $G$, $S'=\{z_1,\dots,z_k\}$ a symmetric set of generators for $Z$,
and let $\overline{S}$ be the set of generators for $E$ given by $\{s(x_i),\, i(z_j)\ |\ i=1,\dots,n\, ,\ j=1,\dots,k\}$.
 We also denote by $d_G$, $d_Z$ and $d_E$ the  word metrics on $G$ and $E$ induced by $S$, $S'$ and $\overline{S}$,
respectively.
Since $d_E$--balls of finite radius only contain a finite number of elements of $i(Z)$, in order to show
that $\omega$ is   weakly bounded 
it suffices to show that, for every $h\in G$, the value of  
$$
\|\omega(g,h)\|_{\overline{S}}=\|s(gh)^{-1}s(g)s(h)\|_{\overline{S}}
$$ 
is uniformly bounded as $g$ varies in $G$. 
But, if $k$ is a Lipschitz constant for $s$, then
\begin{align*}
\|s(gh)^{-1}s(g)s(h)\|_{\overline{S}}&\leq \|s(gh)^{-1}s(g)\|_{\overline{S}}+\|s(h)\|_{\overline{S}}=d_E(s(gh),s(g))+\|s(h)\|_{\overline{S}}\\ &\leq 
kd_G(gh,g)+\|s(h)\| _{\overline{S}}=
kd_G(h,1)+\|s(h)\| _{\overline{S}}\ ,
\end{align*}
which is independent of $g$, as required.

On the other hand, suppose that there exists $C\geq 0$ such that $\|\omega(G,x_i)\|_{S'}\leq C$ for every $i=1,\ldots,n$. 
Let $g$ be an element of $G$. 
We show by induction on $\|g\|_S$  that
$$
\|s(g)\| _{\overline{S}} \leq (1+C)\|g\|_{{S}}\ .
$$
The case $\|g\|_S=0$ follows from the fact that $s$ is normalized.
Let us assume the above inequality for all $g'\in G$ with $\|g'\|_S\leq j-1$, and suppose $\|g\|_S=j$.
Then $g=g'x_{i}$ for some $g'\in G$ with $\|g'\|_S= j-1$ and some $x_{i}\in S$. Now
\begin{align*}
\|s(g)\|_{\overline{S}}&=\|s(g')s(x_{i})i({-\omega(g',x_{i}))}\|_{\overline{S}}\\ &\leq
\|s(g')\|_{\overline{S}}+\|s(x_{i})\|_{\overline{S}}+\|i({-\omega(g',x_{i}))}\|_{\overline{S}}\\ &\leq
(1+C)\|g'\|_S+1+\|\omega(G,x_i)\|_{S'}\\ 
&\leq (1+C)(\|g\|_S-1)+1+C = (1+C)\|g\|_S\ .
\end{align*}
By Lemma~\ref{linear:growth} we now have
$$
\|\omega(G,x)\|_{S'}\leq 2C\|x\|_{S}\qquad \text{for every}\ x\in G\ .
$$

Take $g,h\in G$ and suppose $d(g,h)=k$. Then $g=hx$, where $\|x\|_S=k$.
We have
$$
s(g)=s(hx)=s(h)s(x)i(-\omega(h,x))\ ,
$$
hence
\begin{align*}
d_E(s(g),s(h))&=d_E(s(h)s(x)i(-\omega(h,x)),s(h))=\|s(x)i(-\omega(h,x))\|_{\overline{S}}\\ 
&\leq \|s(x)\|_{\overline{S}}+\|i(\omega(h,x))\|_{\overline{S}}\leq \|s(x)\|_{\overline{S}}+\|\omega(h,x)\|_{{S}'}\\ 
&\leq
(1+C)\|x\|_S+2C\|x\|_S =
(1+3C)d(g,h)\ .
\end{align*}
This concludes the proof.
\end{proof}

\begin{cor}\label{weak:fund:cor}
A central extension of $G$ by $Z$ is quasi-isometrically trivial if and only if its Euler class is a weakly bounded element of $H^2(G,Z)$.
\end{cor}

The following proposition  shows that, in order to detect Property QITB, it is sufficient to deal with extensions of $G$ by $\Z$.

\begin{prop}\label{ZntoZ}
Let $G$ be a finitely generated group. Then 
every weakly
bounded class in $H^2(G,\Z)$ is bounded if and only if 
every weakly bounded class in $H^2(G,Z)$ is bounded 
for every  finitely generated abelian group $Z$.
 \end{prop}
\begin{proof}
The ``only if'' part of the statement is obvious. Let then $Z$ be a finitely generated abelian group, 
and assume that every weakly
bounded class in $H^2(G,\Z)$ is bounded.
Let us consider a weakly bounded cocycle $\omega\in C^2(G,Z)$.
We have $Z\cong \Z^k\oplus F$, where $F$ is a finite abelian group, hence we may consider $\omega$ as a map
$$\omega\colon G^2\to \Z^k\oplus F\, ,\qquad \omega=(\omega_1,\dots,\omega_k,\omega_F)\ ,$$ 
where $\omega_i(g_1,g_2)$ (resp.~$\omega_F(g_1,g_2)$) is the $i$-th component of the projection of $\omega(g_1,g_2)$ onto $\Z^k$
(resp,~the projection of $\omega(g_1,g_2)$ onto $F$). Since $\omega$ is weakly bounded, every $\omega_i$ is also weakly bounded. 
Under our assumptions, 
this implies that $\omega_i$ is cobordant to a bounded cocycle $\omega'_i\in C^2(G,\Z)$ for every $i=1,\dots,k$. 
This easily implies that the cocycle $\omega$ is cobordant to the cocycle $\omega'=(\omega_1',\dots,\omega_k',\omega_F)$
in $C^2(G,Z)$. But $\omega'$ is bounded, and this concludes the proof. 
\end{proof}

Putting together Proposition~\ref{ZntoZ} and Corollary~\ref{weak:fund:cor} we get the following:

\begin{cor}\label{bastaZ}
Let $G$ be a finitely generated group. Then the following conditions are equivalent:
\begin{enumerate}
\item $G$ satisfies Property QITB.
\item Every weakly bounded class in $H^2(G,\Z)$ is bounded.
\end{enumerate}
\end{cor}

\subsection*{Integral vs.~real coefficients}
Bounded cohomology with real coefficients is better understood than bounded cohomology with integral coefficients (for example, for amenable groups
the real bounded cohomology vanishes, while bounded cohomology with integral coefficients may be non-trivial). Therefore, 
before proceeding with our investigation of Property QITB, we first point out the following
results, which sometimes will allow us to work with real coefficients, rather than with  integral ones.

\begin{lemma}\label{ZtoR1}
Let $\alpha\in H^n(G,\Z)$, and denote by $\alpha_\R$ the image of $\alpha$ in $H^n(G,\R)$ under the change of coefficients map. Then:
\begin{enumerate}
\item $\alpha$ is weakly bounded if and only if $\alpha_\R$ is weakly bounded;
\item $\alpha$ is bounded if and only if $\alpha_\R$ is bounded.
\end{enumerate}
\end{lemma}
\begin{proof}
Statement (2) is proved in~\cite[Theorem 15]{Mineyev1} and in~\cite[Proposition 2.18]{frigerio:book}, and the very same argument
also implies (1).


\end{proof}

\begin{cor}\label{ZtoR}
For every $n\in\mathbb{N}$, if every weakly bounded class in $H^n(G,\R)$ is bounded, then
every weakly bounded class in $H^n(G,\Z)$ is bounded and $G$ satisfies Property QITB.
\end{cor}
\begin{proof}
Let us denote by $\psi\colon H^n(G,\Z)\to H^n(G,\R)$ the change of coefficients map, and
%
%
let $\beta\in H^n(G,\Z)$ be weakly bounded. Then $\psi(\beta)\in H^n(G,\R)$ is weakly bounded, hence bounded by (2).
Lemma~\ref{ZtoR1} now implies that $\beta$ is bounded, as desired. The last assertion now follows from Corollary~\ref{bastaZ}. 
\end{proof}

\section{$\ell^\infty$-cohomology}\label{linfty:sec}

As mentioned in the introduction, weakly bounded classes may be characterized in terms of the so--called \emph{$\ell^\infty$--cohomology of $G$}, which 
we are now going to define. 

 Let $A$ be either a finitely generated abelian group, or the field of real numbers, and let
$\ell^\infty(G,A)$ be the module of bounded functions from $G$ to $A$ (here in the case where $A$ is finitely generated abelian, we say that a function is bounded if and only if it takes finitely many values). We can endow $\ell^\infty(G,A)$ with the structure of a 
left $G$-module via the left action defined by
$$
(g\cdot f)(h)=f(g^{-1}h)\, ,\qquad f\in\ell^\infty(G,A)\, ,\quad g,h\in G\ .
$$
We then denote by $C^*_{(\infty)}(G,A)$ the  cochain complex $C^*(G,\ell^\infty(G,A))$, and we define the $\ell^\infty$-cohomology 
$
H^*_{(\infty)}(G,A)
$
of $G$ as the cohomology of the complex $C^*_{(\infty)}(G,A)$ (we refer the reader e.g.~to~\cite{brown} for the definition of cohomology with twisted coefficients).

If we consider $A$ as a trivial $G$-module, then we can equivariantly embed $A$ into the submodule of $\ell^\infty(G,A)$ 
given by the constant maps.
This map induces a chain map $\iota^*\colon C^*(G,A)\to C_{(\infty)}^*(G,A)$, which defines in turn a map
$$
\iota^*\colon H^*(G,A)\to H_{(\infty)}^*(G,A)\ .
$$

We are now ready to prove Proposition~\ref{fund:prop} from the introduction, which we recall here for the convenience of the reader.
As mentioned in the introduction, Corollary~\ref{weak:fund:cor} and Proposition~\ref{fund:prop} 
allow us to recover Kleiner and Leeb's characterization
via $\ell^\infty$-cohomology
 of quasi-isometrically trivial central extensions. 

\begin{prop}\label{weakly:linf}
Let $\alpha\in H^n(G,A)$, $n\geq 2$. Then $\alpha$ is   weakly bounded if and only if $\iota^n(\alpha)=0$.
\end{prop}
\begin{proof}
Let $\omega\in C^n(G,A)$ be a representative of $\alpha$, and suppose $\iota^n(\alpha)=0$. This means that there exists a 
cochain $\varphi\in C_{(\infty)}^{n-1}(G,A)$, i.e.~a  map $\varphi\colon G^{n-1}\to \ell^\infty(G,A)$, such that, for every $g_1,\ldots,g_n,h\in G$,
\begin{align*}
\omega(g_1,\ldots,g_n)&=(\delta \varphi)(g_1,\ldots,g_n)(h)\\ &=g_1\cdot (\varphi(g_2,\ldots,g_n))(h)+\sum_{i=1}^{n-1} (-1)^i\varphi(g_1,\ldots,g_ig_{i+1},\ldots, g_n)(h)
\\ &+ (-1)^n \varphi(g_1,\ldots,g_{n-1})(h)\\ &=
\varphi(g_2,\ldots,g_n)(g_1^{-1}h)+
\sum_{i=1}^{n-1} (-1)^i\varphi(g_1,\ldots,g_ig_{i+1},\ldots, g_n)(h)
\\ &+ (-1)^n \varphi(g_1,\ldots,g_{n-1})(h)\ .
\end{align*} By setting $h=1$, we obtain
\begin{align*}
\omega(g_1,\ldots,g_n)&=\varphi(g_2,\ldots,g_n)(g_1^{-1})+\sum_{i=1}^{n-1} (-1)^i\varphi(g_1,\ldots,g_ig_{i+1},\ldots, g_n)(1)
\\ &+ (-1)^n \varphi(g_1,\ldots,g_{n-1})(1)\ .
\end{align*}
Therefore, if we set 
$$
f\in C^{n-1}(G,A)\, ,\qquad f(g_1,\ldots,g_{n-1})=-\varphi(g_1,\ldots,g_{n-1})(1)\ ,
$$
then 
\begin{align*}
|(\omega+\delta f)(g_1,\ldots,g_n)|&=|\varphi(g_2,\ldots,g_n)(g_1^{-1})-\varphi(g_2,\ldots,g_n)(1)|\\ &\leq 2\|\varphi(g_2,\ldots,g_n)\|_\infty\ .
\end{align*}
Hence $$|(\omega+\delta f)(G,g_2,\ldots,g_n)|\leq 2\|\varphi(g_2,\ldots,g_n)\|_\infty<+\infty\ ,$$ and $\alpha$ is   weakly bounded.

Suppose now that $\omega\in C^n(G,A)$ is a   weakly bounded representative of $\alpha$, and set
$$
\varphi\colon G^{n-1}\to \ell^\infty(G,A)\, ,\qquad \varphi(g_1,\ldots,g_{n-1})(h)=\omega(h^{-1},g_1,\ldots,g_{n-1})
$$
(the fact that $\varphi$ is well-defined is due to the   weak boundedness of $\omega$).
Then, for every $g_1,\ldots,g_n,h\in G$ we have
\begin{align*}
(\delta \varphi)(g_1,\ldots,g_n)(h)&=
\omega(h^{-1}g_1,g_2,\ldots,g_n)+\sum_{i=1}^{n-1} (-1)^i\omega(h^{-1}, g_1,\ldots,g_ig_{i+1},\ldots, g_n)
\\ &+ (-1)^n \omega(h^{-1},g_1,\ldots,g_{n-1})\\ & =\omega(g_1,\ldots,g_n)\ ,
\end{align*}
where the last equality is due to the fact that $\omega$ is a cocycle. Thus $\iota^n(\alpha)=0$, as desired.
\end{proof}


\subsection*{Higher degrees}
As we explain below, the following result readily implies Proposition~\ref{many:counterexamples} from the introduction. 
\begin{prop}\label{pd-prop}
Let $A=\R$ or $\Z$,
let $G$ be an $n$-dimensional non-amenable Poincar\'e duality group, and let $G'=G\times \Z$. Then the sequence
$$
\xymatrix{
 H^{n+1}_b(G',A) \ar[r]^{c^{n+1}} & H^{n+1}(G',A) \ar[r]^{\iota^{n+1}} & H^{n+1}_{(\infty)}(G',A) 
 }
$$
is not exact.
\end{prop}
\begin{proof}
By Lemma~\ref{ZtoR1}  and Proposition~\ref{weakly:linf}, it is sufficient to prove the proposition for $A=\Z$.

We first show that 
\begin{equation}\label{claim1}
H^{n+1}_{(\infty)}(G',\Z)=0\ .
\end{equation}
Since $G'$ is  an $(n+1)$-dimensional Poincar\'e duality group,
we have 
$$
H^{n+1}_{(\infty)}(G',\R)=H^{n+1} (G',\ell^\infty(G',\R))\cong H_0(G',\ell^\infty(G',\R))\ .
$$
It was first observed in~\cite{Niblo} that 
 $H_*(G',\ell^\infty(G',\R))$ is isomorphic to the so-called \emph{uniformly finite homology} of $G'$
(see also~\cite{Diana1,Diana2}). By a fundamental result by Block and Weinberger, amenable groups
can be characterized as those groups for which  uniformly finite homology does not vanish~\cite{BW} in degree 0. Since $G'$ contains the non-amenable
group $G$ as a subgroup, it is itself non-amenable, hence 
$H^{n+1}_{(\infty)}(G',\R)=0$. In order to get~\eqref{claim1} it is now sufficient to recall 
from~\cite[Proposition 7.2]{Gersten-lin-isop} that $H^{n+1}_{(\infty)}(G',\Z)\cong H^{n+1}_{(\infty)}(G',\R)$.
As usual, Gersten proved this fact under the assumption that $G'$ admits a finite $n$-skeleton. For the general case, observe 
that the short exact sequence
$0\to \mathbb{Z}\to \mathbb{R}\to \mathbb{R}/\mathbb{Z}\to 0$ of coefficients induces the short exact sequence of $G'$-modules 
$
0\to \ell^\infty(G',\mathbb{Z})\to \ell^\infty (G',\mathbb{R})\to \ell^\infty (G', \mathbb{R}/\mathbb{Z})\to 0,
$
hence the short exact sequence of complexes
\begin{equation}\label{shortcoeff}
0\to C^*_{(\infty)}(G',\mathbb{Z})\to C^*_{(\infty)}(G',\mathbb{R})\to C^*_{(\infty)}(G',\mathbb{R}/\mathbb{Z})\to 0\ .
\end{equation}
Observe now that $\ell^\infty(G',\mathbb{R}/\mathbb{Z})\cong \text{Hom}(\mathbb{Z}G', \mathbb{R}/\mathbb{Z})$ as $G'$-modules, and recall
that $\text{Hom}(\mathbb{Z}G', \mathbb{R}/\mathbb{Z})$ is an injective $G'$-module (see e.g.~\cite[Proposition 6.1]{brown}).
Therefore, $H^n_{(\infty)}(G',\mathbb{R}/\mathbb{Z})=H^n(G', \ell^\infty(G,\mathbb{R}/\mathbb{Z}))=0$ for every $n\geq 1$. By looking at the long exact sequence
in cohomology induced by~\eqref{shortcoeff} we may now conclude that $H^{n}_{(\infty)}(G',\Z)\cong H^{n}_{(\infty)}(G',\R)$ for every $n\geq 1$.

Let now $\alpha\in H^{n+1}(G',\Z)$ be bounded. 
The group homomorphism
$h\colon G\times \Z\to G\times \Z$, $h(g,m)=(g,2m)$ induces the multiplication by $2$ on $H^{n+1}(G',\Z)$. Since maps induced by homomorphisms do not increase the seminorm
of cohomology classes, this implies that $\|\alpha\|_\infty=0$. Thus $\langle \alpha ,\beta\rangle=0$ for every $\beta\in H_{n+1}(G',\Z)$.
Since $G'$ is an $(n+1)$-Poincar\'e duality group, this implies in turn that
 $\alpha=0$. We have thus shown that both maps in the sequence of the statement (with $A=\Z$) are null. Since $H^{n+1}(G',\Z)\neq 0$,
 the sequence is not exact.
\end{proof}

The following corollary implies Proposition~\ref{many:counterexamples} from the introduction.

\begin{cor}
Let $A=\Z$ or $\R$.
For every $n\geq 3$, let $G_n=\G_2\times \Z^{n-2}$, where $\G_2$ is the fundamental group of the closed oriented surface of genus 2. Then the sequence 
$$
\xymatrix{
 H^{n}_b(G_n,A) \ar[r]^{c^{n}} & H^{n}(G_n,A) \ar[r]^{\iota^{n}} & H^{n}_{(\infty)}(G_n,A) 
 }
$$
is not exact. 
\end{cor}
\begin{proof}
We can apply the previous proposition to the non-amenable $(n-1)$-dimensional Poincar\'e duality group $G=\G_2\times \Z^{n-3}$.
\end{proof}

\subsection*{Property QITB and Gromov's Conjecture}
The strategy described in Proposition~\ref{pd-prop} cannot be implemented in degree 2. Indeed, as stated in the introduction we have the following:

\begin{varthm}[Theorem~\ref{equiv:thm}]
Let $V$ be a compact  Riemannian manifold.
Then the following are equivalent:
\begin{enumerate}
\item A class $\alpha\in H^2(V,\R)$
 is $\widetilde{d}$-bounded if and only if it is bounded.
\item The sequence
$$
 \xymatrix{
 H^2_b(\pi_1(V),\R) \ar[r]^{c^2} & H^2(\pi_1(V),\R) \ar[r]^{\iota^2} & H^n_{(\infty)}(\pi_1(V),\R)
 }
 $$
is exact.
\end{enumerate}
In particular, by Corollary~\ref{ZtoR} and Proposition~\ref{weakly:linf}, if $V$ satisfies Gromov's Conjecture~\ref{Gromov_conj}, then
 $\pi_1(V)$ satisfies Property QITB.
\end{varthm}
\begin{proof}
We denote by $\Omega_\flat^*(\widetilde{V})$ the space of bounded differential forms on $\widetilde{V}$ with bounded differential, and 
by $H_\flat^*(\widetilde{V})$ the associated cohomology (caveat: $H_\flat^*(\widetilde{V})$ is not at all equal to $H_b^*(\widetilde{V},\R)$; indeed, 
since $\widetilde{V}$ is simply connected, $H_b^n(\widetilde{V},\R)=0$ for every $n\geq 1$~\cite[page 40]{Gromov}, \cite[Theorem 2.4]{Ivanov}, \cite[Corollary 4]{FriMoM}, while we will see below that 
$H_\flat^*(\widetilde{V})$ is isomorphic to $H^*_{(\infty)}(\widetilde{V},\R)$, which is often non-trivial in positive degree).

 Observe that, if $\omega$ is a $k$-differential form
on $V$, then the pull-back $\widetilde{\omega}$ to $\widetilde{V}$ is equivariant with respect to a cocompact action, hence it belongs to $\Omega_\flat^*(\widetilde{V})$. 
We thus have a map $\psi\colon H^*(V)\to H^*_\flat (\widetilde{V})$, where $H^*(V)$ denotes the usual de Rham cohomology of $V$. Via the identification $H^*(V)\cong H^*(V,\R)$
due to de Rham isomorphism, the kernel
of $\psi$ coincides with the space of $\widetilde{d}$-bounded classes defined in the introduction, by definition of $\widetilde{d}$-bounded.


Every smooth manifold admits a PL-structure~\cite{Whitehead, Munkres}, hence $V$ is homeomorphic to the geometric realization of a simplicial complex. If we endow the universal covering $\widetilde{V}$
with the induced simplicial structure, integration provides a  chain map $\Omega_\flat^*(\widetilde{V})\to C^*_{(\infty)} (\widetilde{V},\R)$, where $C^*_{(\infty)}(\widetilde{V},\R)$ denotes
here the space of 
bounded cellular cochains on $\widetilde{V}$. If we denote by $H^*_{(\infty)}(\widetilde{V},\R)$ the cohomology of $C^*_{(\infty)}(\widetilde{V},\R)$, we thus get a map
$$
I^*\colon H^*_\flat(\widetilde{V})\to H^*_{(\infty)}(\widetilde{V},\R)\ .
$$
It is shown in~\cite[Theorem 1.11]{Milizia} (see also \cite[Theorem 3.1]{mineyevell1}) that the map $I^*$ is 
an isomorphism in every degree. Moreover, 
by lifting cellular cochains from $V$ to $\widetilde{V}$ we get a well-defined map $\iota^*_V\colon H^*(V,\R)\to H^*_{(\infty)}(\widetilde{V},\R)$.

It readily follows
from the explicit description of the de Rham isomorphism that the following diagram is commutative:
$$
\xymatrix{
H^*(V) \ar[r]^{\psi} \ar[d]^{\cong}& H^*_\flat(\widetilde{V})\ar[d]^{\cong}\\
H^*(V,\R) \ar[r]^{\iota^*_V} & H^*_{(\infty)}(\widetilde{V},\R)\ ,
}
$$
where the vertical arrows correspond to the de Rham isomorphism and to the map $I^*$.
Therefore, a class $\alpha\in H^*(V,\R)$ is $\widetilde{d}$-bounded if and only if it belongs to the kernel
of $\iota^*_V$, and condition (1) is equivalent to the exactness of the sequence
$$
\xymatrix{
H^2_b(V,\R) \ar[rr]^{c^2_V} & & H^2(V,\R) \ar[rr]^{\iota^2_V} & & H^2_{(\infty)}(\widetilde{V},\R)\ .
}
$$

Let us now set $G=\pi_1(V)$, and 
let $f\colon V\to X$ be a classifying map, where $X$ is a $K(G,1)$-space.  We then have canonical 
identifications $H^*(X,\R)\cong H^*(G,\R)$ and $H^*_b(X,\R)\cong H^*_b(G,\R)$ (see e.g.~\cite[Chapter 5]{frigerio:book}). 
We may assume that $V$ and $X$ share the same $2$-skeleton. In particular, the $2$-skeleton of $X$ is finite, and 
we have  a canonical identification $H^*_{(\infty)}(\widetilde{X},\R)= H^*_{(\infty)}(G,\R)$ (see e.g.~\cite{Wie,blank:thesis, Milizia}; indeed, in Gersten's original approach
the module $H^*_{(\infty)} (G,\R)$ was \emph{defined} via this identification).
We thus have the following commutative diagram, where the vertical arrows are all induced by $f\colon V\to X$:
\begin{equation}\label{diagramma}
\xymatrix{
 H^2_b(G,\R) \ar[r]^{c^2_G} \ar[d]^{f_b^*}& H^2(G,\R) \ar[r]^{\iota^2_G} \ar[d]^{f^*}& H^2_{(\infty)}(G,\R) \ar[d]^{f_{(\infty)}^*}\\
 H^2_b(V,\R) \ar[r]^{c^2_V} & H^2(V,\R) \ar[r]^{\iota^2_V} & H^2_{(\infty)}(\widetilde{V},\R)\ .
}
\end{equation}

Using that $V$ and $X$ share the same $2$-skeleton (and the fact that singular homology  may be computed
via cellular cochains), it is easy to show that $f^*$ and $f^*_{(\infty)}$ are both injective.

(2) $\Rightarrow$ (1):  Let us suppose the top row of the diagram is exact, 
and take an element $\alpha\in H^2(V,\R)$ with $\iota^2_V(\alpha)=0$. Since the pull-back of $\alpha$ to $\widetilde{V}$
is null (as an ordinary cohomology class), $\alpha$ vanishes on every homology class which may be represented by a sphere. As a consequence, there exists $\beta\in H^2(G,\R)$
with $f^*(\beta)=\alpha$.  Together with the injectivity of $f^*_{(\infty)}$, the fact that $\iota^2_V(\alpha)=0$ now implies
that $\iota^2_G(\beta)=0$. We thus have $\beta=c^2_G(\beta_b)$ for some $\beta_b\in H^2_b(G,\R)$. By the commutativity of the left square of the diagram,
this implies that $\alpha$ is bounded, hence condition (2) holds.

\smallskip 

(1) $\Rightarrow$ (2). 
Suppose now that the bottom row of the diagram is exact, and
take  $\beta \in \ker i^2_G \subseteq H^2(G,\R)$. 
Then $f^*(\beta)$ lies
in the image of $c^2_V$. Now a fundamental theorem by Gromov ensures that $f^*_b$ is an isomorphism~\cite[Corollary (A) at page 40]{Gromov},
\cite[Theorem 4.3]{Ivanov}, \cite[Theorem 3]{FriMoM}, and this (together with the injectivity of $f^*$)
implies that $\beta$ is bounded, as desired. 
\end{proof}

\section{A group without Property QITB}\label{counter:sec}
We are now ready to exhibit a finitely generated group $G$ admitting a quasi-isometrically trivial central extension whose Euler
class is not bounded. 

\begin{rem}
 (The idea of the construction.) We briefly and informally describe the group that we will be studying. Start with a free product $\widehat G$ of 
 countably many copies of the fundamental group of the genus-2 surface. Any cohomology class of degree 2 on $\widehat G$ that takes unbounded values on the surfaces cannot be bounded. However, all these classes are weakly bounded, the rough reason being the following. Each such class can be represented by the (infinite) sum $\omega$ of pull-backs of $2$-cocycles on the copies of the fundamental 
 group of the genus-2 surface, where we take the pull-back under the natural projection to a factor of the free product. Since the genus-2 surface is hyperbolic, these $2$-cocycles may be chosen to be bounded. Moreover, for any pair $(g_1,g_2)$ of elements
 of $\widehat G$, only a finite number of projections of $g_1$ and $g_2$ on the free factors of $\widehat{G}$ are non-trivial. Putting all these facts together it is easy to show that the set
 $\omega(\widehat{G},g_1,g_2)$ is bounded, hence every degree-2 class on $\widehat{G}$ is weakly bounded.

  The idea is then to ``make $\widehat G$ finitely generated''. This can be done by adding ``stable letters'' $t_i$ to $\widehat G$ that conjugate the generators of each of the surface groups to the generators of the ``next'' surface group. In this way, the generators of the first surface group and the stable letters suffice to generate the new group $G$. Notice that $\widehat G$ is the fundamental group of a locally CAT(0) complex obtained by gluing surfaces and squares (we do not need this fact), which is one way to control the (co)homology of $G$. The idea is that some of the cohomology classes on $\widehat G$ that we discussed above should be the pull-back of weakly bounded classes on $G$, and we believe that it is possible to show this using CAT(0) techniques. However, below we actually give a different description of $G$ that will allow us to use small-cancellation techniques instead. (We will not need the equivalence of the two descriptions.)
\end{rem}

\subsection{Description of the group}
\label{subsec:description}
Let us now proceed with the construction. Let $G$ be the finitely generated group described by the 
following (infinite) presentation $\mathcal{P}$:
$$
\langle a_1,a_2,a_3,a_4,t_1,t_2,t_3,t_4\, |\, r_0,r_1,\dots,r_i,\dots\rangle\ ,
$$
$$
r_i=\left[ t_1^ia_1t_1^{-i},t_2^{i}a_2t_2^{-i}\right] \cdot \left[ t_3^ia_3t_3^{-i},t_4^{i}a_4t_4^{-i}\right]\ ,
$$
where $[w_1,w_2]$ denotes the commutator $w_1w_2w_1^{-1}w_2^{-1}$. We denote by
$F_8$ the free group on the generators $a_1,\dots,a_4,t_1,\dots,t_4$, and by
 $N$ the normal closure
of the relations $r_i$ in $F_8$, so that $G\cong F_8/N$.

Notice that all relations are products of two commutators, and therefore are best thought of as yielding a closed surface of genus 2 in the corresponding presentation complex. As in the heuristic explanation above, the standard generators of the fundamental groups of any two such surfaces are conjugate, but the conjugating elements are different so one can expect that the whole fundamental groups are not conjugate to each other. Indeed, we will show that the surfaces yield distinct elements in homology.

\subsection{Properties of the group}
A key property of the group $G$ is that its presentation $\mathcal{P}$ satisfies the following
small cancellation condition:

\begin{lemma}\label{cancellation}
The presentation $\mathcal{P}$ satisfies the $C'(1/7)$ small cancellation condition. Moreover, no $r_i$ is a proper power in $F_8$.
\end{lemma}
\begin{proof}
Observe that $|r_i|=16i+8$. Moreover, if $i< k$, then the longest common pieces shared by $r_i$ and $r_k$ (or by their cyclically conjugate words) 
are of the form $p=t_j^ia_j^{\pm 1}t^{-i}_j$. 
Therefore, $|p|/|r_i|\leq (2i+1)/(16i+8)=1/8<1/7$, and $|p|/|r_k|<|p|/|r_i|<1/7$, and this proves the first statement.

In order to show that no $r_i$ is a proper power, first observe that the elements $ t_1^ia_1t_1^{-i},t_2^{i}a_2t_2^{-i}, t_3^ia_3t_3^{-i},t_4^{i}a_4t_4^{-i},t_1,t_2,t_3,t_4$
provide a free basis of $F_8$ (in fact, they obviously generate $F_8$, which is Hopfian). Therefore, we have $r_i=[x,y]\cdot [z,t]$ for elements $x,y,z,t$ of a free basis of $F_8$.
Let us identify elements of $F_8$ with their reduced form with respect to this basis, and suppose by contradiction $r_i=h^n$ for some $n>1$. Since 
the first letter of $r_i$ is $x$, we have an equality of reduced words $h=x^m w x^{-m'}$ for some $m\geq 1$ (while possibly $m'=0$) and some word $w$ in $x,y,z,t$. We first
show that $m=m'$. In fact, if $m>m'$ (resp.~$m<m'$), then the reduced form of $r_i=h^n$ would contain at least $n$ occurrences of $x$ with a positive (resp.~negative) exponent, against
the fact that $r_i=[x,y]\cdot [z,t]$. Hence $m=m'$ and $r_i=h^n=x^mw' x^{-m}$, where $w'$ is the reduced form of $w^n$. In particular, the last letter
of the reduced form of $r_i$ should be equal to $x^{-1}$, against the fact that $r_i=[x,y]\cdot [z,t]$. This concludes the proof.
\end{proof}

Let us now describe the second cohomology group of $G$. We denote by $X$ the presentation complex associated to $\mathcal{P}$, that is, the 
CW-complex having exactly one vertex, one edge for each generator of $\mathcal{P}$, and one $2$-cell for each relation of $\mathcal{P}$ (with attaching map prescribed by
the corresponding relation). By construction we have $\pi_1(X)\cong G$. Moreover,
since the relators of the presentation $\mathcal{P}$ are not  proper powers in $F_8$, and $\mathcal{P}$ 
satisfies the $C'(1/7)$-cancellation property,  $X$ is aspherical (see e.g. \cite[Theorem 13.3]{Olshanskii-relations}), and we have a canonical isomorphism $H^2(G,\Z)\cong H^2_{\text{cell}}(X,\Z)$.

We compute the (co)homology of $X$ (hence, of $G$) via cellular (co)chains. Let $c_i$ be the cell
corresponding to the relation $r_i$. By construction, the boundary of $c_i$ is zero, hence $c_i$ is a cycle which defines a class $[c_i]\in H_2^{\text{cell}}(X,\Z)$. 
For every $\alpha\in H^2_{\text{cell}}(X,\Z)$ we set $\alpha_i=\langle\alpha,[c_i]\rangle$, where $$\langle\cdot,\cdot\rangle \colon H^2_{\text{cell}}(X,\Z)\times H_2^{\text{cell}}(X,\Z)\to\Z$$ denotes the Kronecker pairing.
We then set 
$$
\psi\colon H^2_{\text{cell}}(X,\Z)\to  \Z^\N\, ,\qquad \psi(\alpha)=(\alpha_0,\dots,\alpha_i,\dots )\ .
$$



Since $\partial c_i=0$ for every $i\in\mathbb{N}$, the space $C^2_{\text{cell}}(X,\Z)$ contains no non-trivial coboundaries (indeed, for every $\varphi\in C^1_{\text{cell}}(X,\Z)$
and every $i\in\mathbb{N}$ we have $\delta\varphi (c_i)=\varphi(\partial c_i)=0$, hence $\delta\varphi=0$). On the other hand,
since $X$ is $2$-dimensional, every $\Z$-valued map on the set of $2$-cells of $X$ is a cocycle. These facts immediately imply the following:

\begin{prop}\label{H2G}
The map $\psi\colon H^2_{\text{cell}}(X,\Z)\to \Z^\N$ is an isomorphism.
\end{prop}

We have already observed that there is a canonical isomorphism $H^2(G,\Z)\cong H^2_{\text{cell}}(X,\Z)$. Henceforth, we will denote by the same symbol cohomology classes
which correspond to each other under this isomorphism.

\begin{prop}\label{boundedbounded}
Let $\alpha\in H^2(G,\Z)$ be bounded. Then the sequence $\psi(\alpha)\in \Z^\N$ is bounded.
\end{prop}
\begin{proof}
Let us denote by $\alpha_\R\in H^2(X,\R)$ the image of $\alpha$ under the change of coefficients map and the canonical isomorphism between cellular and singular homology.
The change of coefficients map $H^2(G,\Z)\to H^2(G,\R)$ is obviously norm non-increasing, and the canonical isomorphism $H^2(G,\R)\cong H^2(X,\R)$ is isometric
(see e.g.~\cite[Theorem 5.5]{frigerio:book}).
As a consequence, we have $\|\alpha_\R\|_\infty\leq \|\alpha\|_\infty$. 

For every $i\in\mathbb{N}$, the relation $r_i$ is the product of two commutators. As a consequence, there exists a 
map $j_i\colon \Sigma_2\to  X$ 
such that $j_i^*(\alpha_\R)=\alpha_i  [\Sigma_2]^*$, where $\Sigma_2$ is the genus-2 surface, and $[\Sigma_2]^*\in H^2(\Sigma_2,\R)$ is the real fundamental coclass of $\Sigma_2$.
A standard duality result between bounded cohomology and $\ell^1$-homology implies that the $\ell^\infty$-norm $\|[\Sigma_2]^*\|_\infty$ of
$[\Sigma_2]^*$ is the inverse of the simplicial volume of the closed surface of genus $2$, which is equal to $4$ (see e.g.~\cite[Proposition 7.10 and Section 8.12]{frigerio:book}). Moreover,
group homomorphisms induce norm non-increasing maps on cohomology. Therefore,
we have 
$$
|\alpha_i|=\frac{ \|j_i^*(\alpha_\R)\|_\infty}{\|[\Sigma_2]^*\|_\infty}=4\|j_i^*(\alpha_\R)\|_\infty\leq 4\|\alpha_\R\|_\infty\leq 4\|\alpha\|_\infty
$$
whence the conclusion. 
\end{proof}

In order to show that the group $G$ does not satisfy Property QITB, we now look for a weakly bounded class $\alpha\in H^2(G,\Z)$ such that
$\psi(\alpha)$ is not bounded. To this aim we first give the following:

\begin{defn}
We say that a class $\alpha\in H^2(G,\Z)$ with $\psi(\alpha)=(\alpha_i)_{i\in\mathbb{N}}$ is \emph{slow} if
$$
\limsup_{i\to +\infty} \frac{|\alpha_i|}{i}<+\infty\ .
$$
If $\alpha$ is slow, then we set
$$
\Lambda(\alpha)=\sup_{i\in\mathbb{N}} \frac{ |\alpha_i|}{2i+1} \ <\ +\infty\ .
$$
\end{defn}
The rest of this section is devoted to the proof of Theorem~\ref{slow}, which states that 
a class in $H^2(G,\Z)$ is weakly bounded if and only if it is slow.
Since there obviously exist slow classes $\alpha\in H^2(G,\Z)$ for which $\psi(\alpha)$ is not bounded (for example, one may take the class
$\alpha$ such that $\psi(\alpha)=(1,2,3,4,\dots)$ or such that $\psi(\alpha)=(\log 1,\log 2,\log 3,\dots)$), Propositions~\ref{boundedbounded} and Theorem~\ref{slow}
imply the existence of 
weakly bounded classes in $H^2(G,\Z)$ which are not bounded, thus showing that $G$ does not satisfy Property QITB.

It is well known that groups admitting a finite $C'(1/7)$  presentation have linear Dehn function, and are therefore word hyperbolic. 
We have shown above that $H^2(G,\Z)$ is not finitely generated, hence our group $G$ does not admit a finite presentation (thus, it cannot be hyperbolic). 
Nevertheless, we are now going to study
a suitably modified Dehn function for the presentation $\mathcal{P}$. 
We first  define a notion of area which takes the value $2i+1$ on the relation $r_i$.

\begin{defn}\label{area}
Let $w$ be a word in $N<F_8$. We then set
$$
A(w)=\min \left\{\sum_{l=1}^k (2i_l+1)\, \big|\, w=\prod_{l=1}^k w_lr_{i_l}^{\pm 1}w_l^{-1},\ k,i_l\in\mathbb{N},\ w_l\in F_8\right\}\ .
$$
\end{defn}

The proof of the following proposition provides a linear isoperimetric inequality  (with respect to our definition of area) for the infinite presentation $\mathcal{P}$.
 
\begin{prop}\label{isoperimetric}
For every $w\in N$,
$$
A(w)\leq {|w|}\ .
$$
\end{prop}
\begin{proof}
We prove the statement by induction on $|w|$, the case $|w|=0$ being obvious. Thus, let $|w|>0$. By the Greendlinger's Lemma, there exist a subword $w_0$ of $w$ 
and a cyclic permutation $r'_i$ of a relation $r_i$ such that $w_0$ is also an initial subword of $r_i'$, and $|w_0|>(1-3/7)|r_i|=(4/7)|r_i|$. 
 We thus have $r'_i=w_0v$, $w=uw_0z$, and 
 $|v|=|r'_i|-|w_0|<(3/7)|r'_i|$. We also have 
 $
 w=uw_0z=ur'_iv^{-1}z
 $, and, since $r'_i$ is conjugate to $r_i$ and $w$ is conjugate to $r'_iv^{-1}zu$, 
 $$
 A(w)=A(r'_iv^{-1}zu)\leq 2i+1+ A(v^{-1}zu)\ .
 $$
 But $$|v^{-1}zu|\leq |w|-|w_0|+|v|\leq |w|-(1/7)|r_i|=|w|-(16i+8)/7<|w|-2i-1\ ,$$
 hence $A(v^{-1}zu)\leq |w|-2i-1$ by our inductive hypothesis, and $A(w)\leq |w|$, as desired.
\end{proof}

It is proved in~\cite{Milizia} that, in the context of cellular cohomology with real coefficients, weakly bounded classes may be characterized as those classes which admit representatives satisfying a
linear isoperimetric inequality, in the following sense:

\begin{defn}
Let $\widetilde{X}$ be the universal covering of $X$, endowed with the induced cellular structure. 
For every $n\in\mathbb{N}$, let us endow the cellular chain module $C_n^{\text{cell}} (\widetilde{X},\R)$ with the $\ell^1$-norm defined by
$$\big\|\sum_{c\in S_n} a_c\cdot c\big\|_1=\sum_{c\in S_n} |a_c|\ ,$$ where $S_n$ is the set of $n$-cells of $\widetilde{X}$. 
Let $\alpha\in H^n_{\text{cell}}(X,\R)$. Then $\alpha$ satisfies a linear isoperimetric inequality if it admits a representative $z\in C^n_{\text{cell}}(X,\R)$ such that the following condition
holds:
if $\widetilde{z}\in C^n_{\text{cell}}(\widetilde{X},\R)$ is the pull-back of $z$ to $\widetilde{X}$, then
 there exists a constant $K\geq 0$  such that
$$
|z(c)|\leq K\cdot \|\partial c\|_1\qquad \text{for every}\ c\in C^{\text{cell}}_{n}(X,\R)\ .
$$
\end{defn}

\begin{prop}\label{Milizia:prop}
Let $G$ and $X$ be respectively the group and the presentation complex associated to the presentation $\mathcal{P}$ described in Subsection \ref{subsec:description}.

Let $\alpha\in H^2(G,\Z)$ and let $\alpha_\R\in H^2_{\text{cell}}(X,\R)$ be the corresponding class in the cellular homology of $X$ with real coefficients.
Then $\alpha$ is weakly bounded if and only if $\alpha_\R$ satisfies a linear isoperimetric inequality.
\end{prop}
\begin{proof}
By Lemma~\ref{ZtoR1}, the class $\alpha$ is weakly bounded if and only if the corresponding class $\alpha_\R$ in $H^2(G,\R)$ is so.
Moreover, Proposition~\ref{weakly:linf} shows that the set of weakly bounded classes in $H^2(G,\R)$ coincides with the kernel of
the map $\iota^2\colon H^2(G,\R)\to  H_{(\infty)}^*(G,\R)$
described in Section~\ref{linfty:sec}.
Now the conclusion follows from~\cite[Theorem 1.8]{Milizia}, together with the fact (also proved in~\cite{Milizia}) that 
 the following diagram is commutative:
\[\xymatrix{
    H^*(G,\R) \ar[r]^{\iota^*} \ar[d]^{\cong} & H_{(\infty)}^*(G,\R) \ar[d]^{\cong} \\
    H_{\text{cell}}^*(X,\R) \ar[r]^{\iota_X^*} & H_{(\infty)}^*(\widetilde{X},\R)\ .
}\]
\end{proof}

As explained above, the following theorem concludes the proof of Theorem~\ref{counterexample:thm} from the introduction.

\begin{thm}\label{slow}
Let $\alpha\in H^2(G,\Z)$. Then $\alpha$ is weakly bounded if and only if it is slow.
\end{thm}
\begin{proof}
By Proposition~\ref{Milizia:prop}, we need to show that the class $\alpha_\R\in H^2_{\text{cell}}(X,\R)$ corresponding to $\alpha$ 
satisfies a linear isoperimetric inequality if and only if $\Lambda (\alpha)<+\infty$. 

We have already observed that there are no coboundaries in $C^2_{\text{cell}}(X,\R)$, hence the class $\alpha_\R$ admits a unique representative
in $C^2_{\text{cell}}(X,\R)$, namely the function (still denoted by $\alpha_\R$) which takes the value $\alpha_i$ on the cell $c_i$. We denote by $\widetilde{\alpha}_\R$ the pull-back
of $\alpha_\R$ to $\widetilde{X}$. Moreover, for every $i\in\mathbb{N}$ we fix a lift $\widetilde{c}_i$ of $c_i$ to $\widetilde{X}$. 

Using that  the relation
$r_i$ has length $16i+8$, it is not difficult to show that, if $\widetilde{c}_i$ is a lift of $c_i$ to $\widetilde{X}$, then $\|\partial \widetilde{c}_i\|_1=16i+8$. Therefore,
if $\alpha_\R$ satisfies a linear isoperimetric inequality, then there exists $K\geq 0$ such that
$$
|\alpha_i|=|\alpha_\R (c_i)|=|\widetilde{\alpha}_\R(\widetilde{c}_i)|\leq K\cdot (16i+8)\ ,
$$
hence $\Lambda(\alpha)<+\infty$. 

Let us now suppose that $\alpha$ is slow, so that there exists $\Lambda\geq 0$ such that $|\alpha_i|\leq \Lambda (2i+1)$ for every $i\in\mathbb{N}$.
Let $\widetilde{c}\in C^2_{\text{cell}}(\widetilde{X},\R)$ be any chain. By~\cite[Theorem 3.3]{AlGe}, we have $\partial \widetilde{c}=\sum_{h=1}^k a_h \gamma_h$, 
where $a_h\in\R$ and $\gamma_h$ is (the integral chain corresponding to) a simple closed loop in the $1$-skeleton of $\widetilde{X}$. 
Moreover, the $\gamma_h$ are \emph{coherent} (according to the terminology of~\cite{AlGe}), which is equivalent to the fact that
 $\|\partial \widetilde{c}\|_1=\sum_{h=1}^k |a_h|\cdot  \|\gamma_h\|_1$.
  Simple closed loops in the 1-skeleton of $\widetilde{X}$
correspond to words in $N<F_8$, while (conjugates of) relations in $\mathcal{P}$ correspond
to (translates of) cells $\widetilde{c}_i$ in $\widetilde{X}$. Under this correspondence, Proposition~\ref{isoperimetric} may be restated as follows: for every 
$h=1,\dots, k$, we have $\gamma_h=\partial b_h$, where
$$
b_h=\sum_{j=1}^{l} g_{j}\cdot \widetilde{c}_{i_j}\ ,
$$
for some $g_1,\dots,g_{l}$ in  $G=\pi_1(X)\cong \text{Aut}(\widetilde{X})$, and $\sum_{j=1}^{l} (2i_j+1)\leq \|\gamma_h\|_1$.
As a consequence,  
$$
|\widetilde{\alpha}_\R(b_h)|=\big| \sum_{j=1}^{l} \widetilde{\alpha}_\R(g_{j}\cdot \widetilde{c}_{i_j}) \big|=\big| \sum_{j=1}^{l} \alpha_\R({c}_{i_j}) \big|\leq
\sum_{j=1}^l \Lambda (2i_j+1)\leq \Lambda \|\gamma_h\|_1\ .
$$
By construction, we have $\partial \widetilde{c}=\sum_{h=1}^k a_h\partial b_h$, thus the chain $\widetilde{c}-\sum_{h=1}^k a_h b_h$ is a cycle, hence a boundary (recall
that $X$ is aspherical, hence $\widetilde{X}$ is contractible). But 
$\widetilde{X}$ is a $2$-dimensional cellular complex, thus $\widetilde{c}=\sum_{h=1}^k a_h b_h$ and
$$
|\widetilde{\alpha}_\R(\widetilde{c})|=\big| \widetilde{\alpha}_\R \left(\sum_{h=1}^k a_h b_h\right)\big|\leq \sum_{h=1}^k |a_h|\cdot |\widetilde{\alpha}_\R(b_h)|\leq
\sum_{h=1}^k \Lambda |a_h|\cdot \|\gamma_h\|_1=
\Lambda \|\partial \widetilde{c}\|_1\ .
$$

This shows that $\alpha_\R$ satisfies a linear isoperimetric inequality, and concludes the proof of Theorem~\ref{slow}.
\end{proof}

%
%
%
%

\section{Groups with Property QITB}\label{families:sec}

The main goal of this section is to prove Theorem \ref{families}. We prove (1) in Corollary \ref{cor:amenable_QITB}, (2) in Theorem \ref{thm:rel_hyp_QITB}, (3) follows from Corollary \ref{**QITB} together with Remark \ref{rem:raag_toral}, and (4) is Theorem \ref{3manifolds}.

\subsection*{Amenable groups}
We are now ready to prove that amenable groups satisfy Property QITB.

\begin{prop}\label{amenable:prop}
Let $G$ be an amenable group, and let $\alpha\in H^n(G,\R)$
be weakly bounded. Then $\alpha=0$.
\end{prop}
\begin{proof}
Let $\mu$ be a right-invariant mean on $\ell^\infty(G,\R)$.
Let $\omega\in C^n(G,\R)$ be a weakly bounded representative of $\alpha$, and define 
$f\in C^{n-1}(G,\R)$ as follows. For every $(g_2,\dots,g_{n})\in G^{n-1}$, the function
$$h_{(g_2,\dots,g_n)}\colon G\to \R\, ,\qquad h_{(g_2,\dots,g_n)}(g_1)=\omega(g_1,g_2,\dots,g_n)
$$
is bounded. We may thus set
$$
f(g_2,\dots, g_n)=\mu(h_{(g_2,\dots,g_n)})\qquad \text{for every}\ (g_2,\dots,g_n)\in G^{n-1}\ .
$$
For every $g\in G$ let us denote by $r_g\colon G\to G$ the right multiplication by $G$.
Since $\omega$ is a cocycle, for every $(g_2,\dots,g_{n+1})\in G^n$ the function
 $$h_{(g_3,\dots,g_{n+1})}\circ r_{g_2}-\left(\sum_{i=2}^{n+1} (-1)^i h_{(g_2,\dots, g_ig_{i+1},\dots,g_{n+1})}\right) 
-(-1)^{n+1} h_{(g_2,\dots,g_{n})}
$$
takes the constant value $\omega(g_2,\dots,g_{n+1})$ on $G$. Therefore, by using that $\mu$ is linear and right-invariant, we
have
\begin{align*}
& \omega(g_2,\dots,g_{n+1}) \\
=&\mu\left(h_{(g_3,\dots,g_{n+1})}\right)-\left(\sum_{i=2}^{n+1} (-1)^i \mu(h_{(g_2,\dots, g_ig_{i+1},\dots,g_{n+1})})\right) 
-(-1)^{n+1} \mu\left(h_{(g_2,\dots,g_{n})}\right)\\
=& \delta (f)(g_2,\dots,g_{n+1})\ .
\end{align*}
Thus $\omega$ is a coboundary, and $\alpha=0$, as desired.
\end{proof}

Putting together Corollary~\ref{ZtoR} and Proposition~\ref{amenable:prop} we get the following:

\begin{cor}\label{cor:amenable_QITB}
Every amenable group satisfies Property QITB.
\end{cor}

\begin{rem}\label{amen:rem:ger}
Together with Proposition~\ref{fund:prop}, Proposition~\ref{amenable:prop} implies that, if $G$ is amenable, then the map
 $$
 \iota^n\colon H^n(G,\R)\to H^n_{(\infty)}(G,\R)
 $$
 is injective. 
This result was first proved in~\cite[Theorem 10.13]{gersten:preprint} under the assumption that $G$ admits an Eilenberg-MacLane model
with a finite $n$-skeleton, and by Wienhard~\cite[Proposition 5.3]{Wie} in the general case (see also~\cite[Section 6.3]{blank:thesis}).

In fact, Proposition~\ref{amenable:prop} could be deduced  from Proposition~\ref{fund:prop} and \cite[Proposition 5.3]{Wie}. 
We preferred to add here our proof of Proposition~\ref{amenable:prop} because it is very elementary and self-contained.
 \end{rem}

 \subsection*{Other examples}

 Let $G$ be a group. We say that a class  $\alpha\in H_2(G,\R)$ is \emph{amenable} if there exist an amenable group $A$ and a homomorphism
 $f\colon A\to G$ such that $\alpha$ lies in the image of $f_*\colon H_2(A,\R)\to H_2(G,\R)$.
The amenable classes generate a linear subspace of $H_2(G,\R)$ that we denote by $\ha(G,\R)$. It readily follows from the definitions that, if $f\colon G_1\to G_2$ is a homomorphism,
 then $$f_*(\ha(G_1,\R))\subseteq \ha(G_2,\R)\ .$$

\begin{rem}
It is well known that any element $\alpha \in H_2(G,\R)$ is represented by a surface, i.e.~that $\alpha=f_*(\beta)$ for some $\beta\in H_2(\Gamma_g,\R)$, where $\Gamma_g$ is the
fundamental group of the closed connected orientable surface of genus $g$, $g\geq 1$. In fact, one may define the \emph{genus} of a class $\alpha$ as the minimal 
$g\in\N$ such that $\alpha=f_*(\beta)$ for some $\beta\in H_2(\Gamma_g,\R)$. The class $\alpha$ is \emph{toral} if its genus is equal to or smaller than 1.

Observe that, if $K<\Gamma_g$ and $g\geq 2$, then $\ha(K,\R)=0$. In fact, if the index of $K$ in $\Gamma_g$ is infinite, then $K$ is free, and
$H_2(K,\R)=0$ (hence, a fortiori, $\ha(K,\R)=0$), while if the index of $K$ in $\Gamma_g$ is finite, then $K=\Gamma_{g'}$ for some $g'\geq g\geq 2$.
Since any amenable subgroup of $\Gamma_{g'}$ is either trivial of infinite cyclic, and the degree-2 homology with real coefficients of infinite cyclic groups  vanishes, 
also in this case we may conclude that $\ha(K,\R)=0$.
Building on this remark, one
may wonder whether amenable classes defined above should in fact be toral. However, it is shown in~\cite{BarGhys} that, for every $g\in\N$, there exist a nilpotent (hence, amenable) group $N$ and
a class $\alpha\in H_2(N,\R)$ which does not lie in the subspace generated by classes of $H_2(N,\R)$ with genus smaller than $g$.
\end{rem}

Recall that there exists a duality pairing
 $$
 \langle\cdot \, ,\, \cdot\rangle\colon H^2(G,\R)\times H_2(G,\R)\to \R\ .
 $$
We denote by $\ann(\ha(G,\R))\subseteq H^2(G,\R)$ the annihilator of $\ha(G,\R)$ in $H^2(G,\R)$, i.e.~the subspace of coclasses $\varphi\in H^2(G,\R)$ for which
$$
\langle \varphi,\alpha\rangle=0\qquad \forall\alpha\in\ha(G,\R)\ .
$$
Notice that $\varphi\in\ann(\ha(G,\R))$ if and only if $\langle \varphi,\alpha\rangle=0$ for all amenable classes $\alpha$ (amenable classes might form a proper subset of $\ha(G,\R)$).
If $f\colon G_1\to G_2$ is a homomorphism,
 then $$f^*(\ann(\ha(G_2,\R)))\subseteq \ann(\ha(G_1,\R))\ .$$

\begin{defn}
We say that a group $G$ has Property $(*)$ if $\ha(G,\R)=H_2(G,\R)$.  
\end{defn}

Of course, any amenable group has Property $(*)$. 
Interesting non-amenable examples are given by:
\begin{rem}\label{rem:raag_toral}
 If $G$ is a right-angled Artin group, then $G$ admits a classifying space (the Salvetti complex) whose $2$-skeleton is obtained by gluing tori. From this, one can deduce that $H_2(G,\R)$ is generated by toral classes, so that, in particular, we see that right-angled Artin groups have Property $(*)$. 
\end{rem}

\begin{defn}
We say that a group $G$ has Property $(**)$ if every class in $\ann(\ha(G,\R))$ is bounded, i.e.~$$\ann(\ha(G,\R))\subseteq c(H^2_b(G,\R))\ ,$$ where $c\colon H^2_b(G,\R)\to H^2(G,\R)$ is the comparison map.
 \end{defn}

 Of course, if a group $G$ has Property $(*)$ then it also has Property $(**)$. Property (**) is significant in our context due to the following:
 
 \begin{prop}\label{QITA}
 Let $\alpha\in H^2(G,\R)$ be weakly bounded. Then $\alpha\in \ann(\ha(G,\R))$.
\end{prop}
   \begin{proof}
    Let $\alpha\in H^2(G,\R)$ be a weakly bounded class, 
  and let $\beta$ be an amenable class. 
Then there exist an amenable group $A$ and
 a homomorphism $f\colon A\to G$ such that $\beta=f_*(\beta_A)$ for some $\beta_A\in H_2(A,\R)$. Being the pull-back of a weakly bounded class,
 the element $f^*(\alpha)\in H^2(A,\R)$ is weakly bounded itself. Since $A$ is amenable, Proposition~\ref{amenable:prop} ensures that 
 $f^*(\alpha)=0$, hence
  $$
 \langle \alpha,\beta\rangle=\langle \alpha,f_*(\beta_A)\rangle=\langle f^*(\alpha),\beta_A \rangle=0\ . 
 $$
 We have thus shown that $\alpha$ belongs to $\ann(\ha(G,\R))$. 
   \end{proof}
   
   Proposition~\ref{QITA} and Corollary~\ref{ZtoR} readily imply the following:
   
 \begin{cor}\label{**QITB}
 Suppose $G$ has Property $(**)$. Then $G$ satisfies QITB.
  \end{cor}


\begin{thm}\label{thm:rel_hyp_QITB}
Let $G$ be relatively hyperbolic w.r.t. the finite collection of subgroups $\calH=\{H_1,\ldots,H_k\}$, and suppose that every
$H_i$ has Property $(*)$. Then $G$ has Property $(**)$.
\end{thm}
\begin{proof}
Let us consider the commutative diagram
$$
\xymatrix{
 H^2_b(G,\calH,\R) \ar[d]^{c_1}\ar[r]^{i_b}  & H^2_b(G,\R) \ar[d]^{c_2} & \\
 H^2(G,\calH,\R)\ar[r]^{i} &  H^2(G,\R) \ar[r]^(.4){j} & \bigoplus_{i=1}^k H^2(H_i,\R)\ ,	\\
 }
$$
where the bottom row is exact (see e.g.~\cite{Bieri}). 
Let $\alpha\in \ann(\ha(G,\R))$. Since group homomorphisms preserve the annihilators of amenable classes, $j(\alpha)\in \oplus_{i=1}^k\ann(\ha(H_i,\R))$. 
Since each $H_i$ has Property $(*)$, this means that $j(\alpha)=0$, hence $\alpha=i(\beta)$ for some $\beta\in H^2_b(G,\calH,\R)$. Now the comparison map in relative cohomology for relative hyperbolic pairs is surjective~\cite{Franceschini2}, hence $\beta=c_1(\eta)$ for some $\eta\in H^2_b(G,\calH,\R)$, and $\alpha=c_2(i_b(\eta))$ is a bounded class in $H^2(G,\R)$, as desired.
\end{proof}

\begin{cor}
Let $G$ be relatively hyperbolic w.r.t. the finite collection of subgroups $\calH=\{H_1,\ldots,H_k\}$, and suppose that every
$H_i$ is amenable. Then $G$ satisfies Property QITB.
\end{cor}

As already mentioned in the statement of Theorem~\ref{families}, the previous corollary implies that every hyperbolic group satisfies Property QITB. 
In order to prove this fact, however, there is no need to introduce the machinery above: indeed, Neumann and Reeves~\cite{NeuRee}  and Mineyev~\cite{Mineyev1} proved that, if $G$
is hyperbolic, then every class in $H^2(G,\mathbb{Z})$ is bounded, hence $G$ trivially satisfies Property QITB.

 \subsection*{Virtually trivial central extensions}

\begin{defn}\label{defn:virtually_trivial}
We say that a central extension $$\xymatrix {1\ar[r] &Z\ar[r] &E\ar[r]^{\pi} & G\ar[r] &1}$$ is \emph{virtually trivial} if there exists a finite-index subgroup $G'$ of $G$ such that the induced extension
$$\xymatrix {1\ar[r] &Z\ar[r] &\pi^{-1}(G')\ar[r]^{\pi} & G'\ar[r] &1}$$ is trivial.
\end{defn}

\begin{lemma}\label{lem:torsion_iff_VT}
 A central extension by a
finitely generated torsion-free abelian group is virtually trivial if and only if its Euler class has finite order.
\end{lemma}

\begin{proof}
Let 
$\xymatrix {1\ar[r] &Z\ar[r] &E\ar[r]^{\pi} & G\ar[r] &1}$ be a central extension with Euler class $\alpha\in H^2(G,Z)$.

Suppose first that $\alpha$ has  order $n\in\mathbb{N}$, $n>0$, and let
 $\omega$ be a representative of $\alpha$. Then $n\omega=\delta f$ for some $1$-cochain $f\in C^1(G,Z)$. 
 Thus $f$ defines a homomorphism $\widehat f\colon G\to Z/nZ$. Since the target group is finite, the kernel is a finite index subgroup $G'<G$.
 Take $g\in G'$. Then  $f(g)\in nZ$. Since $Z$ is torsion-free, there exists a unique element 
$h(g)=f(g)/n \in Z$, and we have $\omega=\delta h$ on $G'$.
  This implies that the restriction of $\omega$ to $G'$ is a coboundary, which in turn shows that the induced extension of $G'$ is trivial.
 
 Conversely, suppose that there exists a subgroup $G'$ of $G$ of index $n\in\mathbb{N}$, $n>0$, such that the induced extension of $G'$ is trivial. 
 Let $\text{res}\colon H^2(G,Z)\to H^2(G',Z)$ and $\text{trans}\colon H^2(G',Z)\to H^2(G,Z)$ be the restriction and the transfer map, respectively, and recall
 that $\text{trans}\circ \text{res}\colon H^2(G,Z)\to H^2(G,Z)$ is the multiplication by $n$ (see e.g.~\cite[Proposition 9.5]{brown}). We then have $\text{res}(\alpha)=0$, 
 whence
 $n\alpha=\text{trans}(\text{res}(\alpha))=0$, i.e.~$\alpha$ has finite order in $H^2(G,Z)$.
\end{proof}

\begin{lemma}\label{changecoeff}
Let $G$ be a finitely generated group, let $Z$ be a finitely generated abelian group, and let $j\colon H^2(G,Z)\to H^2(G,Z\otimes \R)$ be the 
change of coefficients map induced by the natural map $Z\to Z\otimes \R$. Then $\ker j$ is equal to the torsion subgroup of $H^2(G,Z)$.
\end{lemma}
\begin{proof}
Since $H^2(G,Z\otimes \R)$ is torsion-free, every element of $H^2(G,Z)$ of finite order lies in $	\ker j$. 

Let now $\alpha\in \ker j$. By looking at the commutative diagram
$$
\xymatrix{
 H^2(G,Z)\ar[r]^(.4){\varphi} \ar[d]^j&\text{Hom}_\Z (H_2(G,\Z),Z) \ar[d]\ar[d]^{j'}\\
 H^2(G,Z\otimes \R)\ar[r] &\text{Hom}_\Z (H_2(G,\Z),Z\otimes \R) 
}
$$
we obtain that $\varphi(\alpha)$ lies in $\ker j'$, and this readily implies that $\varphi(\alpha)$ has finite order, i.e.~there exists $n\in\mathbb{Z}\setminus\{0\}$
such that $\varphi(n\alpha)=0$. Indeed, for every $c\in H_2(G,\Z)$ we have that $\phi(\alpha)(c)$ is an element of the kernel of the map $Z\mapsto Z\otimes \R$, and there exists some (uniform) $n$ such that for any such element $z$ we have $nz=0$.

 By the Universal Coefficient Theorem for cohomology we now have the exact sequence
 $$
0 \to \text{Ext}^1(H_1(G,\Z),Z)\to H^2(G,Z)\longrightarrow\text{Hom}_\Z (H_2(G,\Z),Z)\to 0
$$
Since $G$ and $Z$ are finitely generated, $\text{Ext}^1(H_1(G,\Z),Z)$ is a torsion group. Thus $n\alpha$, which belongs to the (the image of) $\text{Ext}^1(H_1(G,\Z),Z)$, has finite order, and hence $\alpha$ has finite order as well.
\end{proof}

Since amenable groups and right-angled Artin groups satisfy Property~(*), the following result implies Theorem~\ref{amenable:vt} from the introduction.

\begin{thm}\label{*virt}
 Suppose that the group $G$ satisfies Property (*). Then a central extension of $G$ by the
finitely generated abelian group $Z$ is quasi-isometrically trivial if and only if its Euler class has finite order. Moreover, if $Z$ is torsion-free then this happens if and only if the extension
is virtually trivial.
\end{thm}
\begin{proof}
Let $Z$ be any finitely generated abelian group. We first prove that a class in $H^2(G,Z)$ is bounded if and only if it has finite order.

By Proposition~\ref{QITA}, any bounded class $\alpha\in H^2(G,\R)$ vanishes on $\ha(G,\R)=H_2(G,\R)$. By the Universal Coefficient Theorem, this implies that $\alpha=0$. Therefore,
the comparison map $c^2\colon H^2_b(G,\R)\to H^2(G,\R)$ is null.  Observe now that $Z\cong \Z^k\oplus F$, where $F$ is finite. Then $Z\otimes \R\cong \R^k$, thus the comparison map
$H^2_b(G,Z\otimes \R)\to H^2(G,Z\otimes \R)$ is null. By looking at the commutative diagram 
$$
\xymatrix{ H^2_b(G,Z)\ar[r] \ar[d] & H^2(G,Z)\ar[d]^j\\
H^2_b(G,Z\otimes \R)\ar[r]  & H^2(G,Z\otimes \R)}
$$
we can then deduce that every bounded class in $H^2(G,Z)$ is contained in $\ker j$. By Lemma~\ref{changecoeff}, we conclude that 
bounded classes have finite order in $H^2(G,Z)$.

On the other hand, if $\alpha\in H^2(G,Z)$ has finite order, then $j(\alpha)=0$ in $H^2(G,Z\otimes \R)$. The very same argument as in the proof of Lemma \ref{ZtoR1}
now shows that $\alpha$ is bounded. 

We have thus shown that  a class in $H^2(G,Z)$ is bounded if and only if it has finite order. By Corollary~\ref{**QITB}, this implies that a central extension of
$G$ by $Z$ is quasi-isometrically trivial if and only if its Euler class has finite order. We conclude applying Lemma \ref{lem:torsion_iff_VT}.
\end{proof}

\begin{rem}\label{rem:Thompson}
Lemma~\ref{lem:torsion_iff_VT} (and Theorem~\ref{*virt}) cannot hold in general for extensions by finitely generated abelian groups with torsion. For example, let us consider Thompson's group $T$. It is shown in~\cite{GS} that $H^1(T,\Z)=0$, $H^2(T,\Z)=\Z^2$. If $n>1$ is any integer, we then deduce from the Universal Coefficient Theorem
that $H^2(T,\Z_n)=\Z_n^2$. In particular, there exists a non-trivial class $\alpha\in H^2(T,\Z_n)$. This class has obviously finite order. Nevertheless, the 
unique finite-index subgroup of $T$ is $T$ itself, thus $\alpha$ does not vanish on any finite-index subgroup of $T$. The central extension of $T$ by $\Z_n$ with Euler class $\alpha$ is
quasi-isometrically trivial (since its Euler class is obviously bounded), but not virtually trivial. 
\end{rem}

\subsection*{(Amalgamated) products of groups with Property QITB} 
We now prove Propositions~\ref{direct:prop} and~\ref{amalgamated:prop} from the introduction.

\begin{varthm}[Proposition \ref{direct:prop}.]
Let $G_1$, $G_2$ be groups satisfying Property QITB. Then
the direct product $G_1\times G_2$ satisfies Property QITB.
\end{varthm}

\begin{proof}
As usual, we prove that every weakly bounded class in $H^2(G,\Z)$ is bounded, under the assumption
that the same condition holds in $H^2(G_i,\Z)$, $i=1,2$.


Let  $\alpha\in H^2(G,\Z)$ be weakly bounded, and denote by $p_i\colon G\to G_i$  the projection, and by $j_i\colon G_i\to G_1\times G_2$
the inclusion. By the K\"unneth formula, we have
$$
H_2(G_1\times G_2,\Z)\cong (H_1(G_1,\Z)\otimes H_1(G_2,\Z))\oplus H_2(G_1,\Z)\oplus H_2(G_2,\Z)\ .
$$
More precisely, any element  $\beta\in H_2(G_1\times G_2,\Z)$ may be decomposed as
$$
\beta=\overline{\beta}+(j_1)_*(\beta_1)+(j_2)_*(\beta_2)\ ,
$$
where $\overline{\beta}$ corresponds to a class in $H_1(G_1,\Z)\otimes H_1(G_2,\Z)$ under the above identification, and
$\beta_i\in H_2(G_i,\mathbb{Z})$ for $i=1,2$. 
Let now $\alpha_\R\in H^2(G_1\times G_2,\R)$, $\overline{\beta}_\R\in H_2(G_1\times G_2,\R)$ be the classes corresponding
to $\alpha$, $\overline{\beta}$ under the change of coefficients homomorphism (thus, $\alpha_\R$ is also weakly bounded). 

Since any class in $H_1(G_i,\Z)$ is the pushforward of a class in $H_1(\Z,\Z)$ via some homomorphism $f_i\colon \Z\to G_i$, the class
$\overline{\beta}$ is the push-forward of an element of $H_2(\Z\times \Z,\Z)$ via a homomorphism $\Z\times \Z\to G_1\times G_2$ of the form $f_1\times f_2$. As a consequence, 
$\overline{\beta}_\R$ is amenable, and $(p_i)_*(\overline{\beta})=0$ for $i=1,2$. Since $\alpha_\R$ is weakly bounded, by 
Proposition~\ref{QITA} we thus get $\langle\alpha,\overline{\beta}\rangle=\langle \alpha_\R,\overline{\beta}_\R\rangle=0$. Moreover, for $i=1,2$ we have
$\langle p_i^*(j_i^*(\alpha)),\overline{\beta}\rangle=\langle j_i^*(\alpha), (p_i)_*(\overline{\beta})\rangle=0$.

Let us now set  $$\alpha'=\alpha-p_1^*(j_1^*(\alpha))-p_2^*(j_2^*(\alpha))\ .$$ 
We will prove  that $\alpha'$ vanishes on $\beta$, by showing that it vanishes
on its summands $\overline{\beta}$, $(j_1)_*(\beta_1)$ and $(j_2)_*(\beta_2)$.
 
The above computations show that
$$
\langle \alpha',\overline{\beta}\rangle=\langle \alpha,\overline{\beta}\rangle-\langle p_1^*(j_1^*(\alpha)), \overline{\beta}\rangle -\langle p_2^*(j_2^*(\alpha)), \overline{\beta}\rangle=0\ .
$$
Moreover, we have
\begin{align*}
\langle \alpha',(j_1)_*(\beta_1)\rangle&=\langle \alpha,(j_1)_*(\beta_1)\rangle-\langle p_1^*(j_1^*(\alpha)), (j_1)_*(\beta_1)\rangle -\langle p_2^*(j_2^*(\alpha)),(j_1)_*(\beta_1)\rangle\\
&=\langle j_1^*(\alpha), \beta_1\rangle-\langle j_1^*(\alpha), (p_1\circ  j_1)_*(\beta_1)\rangle -\langle j_2^*(\alpha),(p_2\circ j_1)_*(\beta_1)\rangle\\
&=\langle j_1^*(\alpha), \beta_1\rangle-\langle j_1^*(\alpha), \beta_1\rangle-\langle j_2^*(\alpha),0\rangle = 0\ ,
\end{align*}
and a similar computation implies that $\langle \alpha',(j_2)_*(\beta_2)\rangle=0$ too.

 We have thus shown that $\alpha=p_1^*(j_1^*(\alpha))+p_2^*(j_2^*(\alpha))+\alpha'$, where $\alpha'\in H^2(G,\Z)$ is such that $\langle \alpha',\beta\rangle=0$ for
every $\beta\in H_2(G,\Z)$. Since $\alpha$ is weakly bounded, $j_i^*(\alpha)\in H^2(G_i,\Z)$ is also weakly bounded, for $i=1,2$.
But $G_i$ satisfies Property QITB, hence $j_i^*(\alpha)$ is bounded for $i=1,2$, and also $p_1^*(j_1^*(\alpha))+p_2^*(j_2^*(\alpha))$ is bounded.

In order to conclude, it is thus sufficient to show that $\alpha'$ is also bounded. However, 
since $\langle \alpha',\beta\rangle=0$ for
every $\beta\in H_2(G,\Z)$, the image of $\alpha'$ in $H^2(G,\R)$ is trivial, hence $\alpha'$ 
is bounded by Lemma~\ref{ZtoR1}. 
\end{proof}

 \begin{defn}\label{transverse:defn}
 An amalgamated product $G=G_1*_H G_2$ is \emph{transverse}
 if, denoting
 $i_1\colon H\to G_1$ and $i_2\colon H\to  G_2$ the inclusions defining the amalgamated product, 
   the map
 $$
 (i_1)_*\oplus (i_2)_*\colon H_1(H,\R)\to H_1(G_1\,\R)\oplus H_1(G_2,\R)$$ 
 is injective.
 \end{defn}

 \begin{varthm}[Proposition \ref{amalgamated:prop}]
Let $G=G_1*_H G_2$ be a \emph{transverse} amalgamated product, where $H$ is amenable. If $G_1,G_2$ satisfy Property QITB, then
$G$ satisfies Property QITB.
\end{varthm}
  \begin{proof}
  Let $\alpha\in H^2(G,\Z)$ be weakly bounded, and denote by $\alpha_\R\in H^2(G,\R)$ the image of $\alpha$ via the change of coefficients map. Then $\alpha_\R$
  is weakly bounded, and by Lemma~\ref{ZtoR1} we are left to show that $\alpha_\R$ is bounded.
  
By definition of transverse amalgamated product, the map $H_1(H,\R)\to H_1(G_1,\R)\oplus H_1(G_2,\R)$ is injective. By looking at the Mayer-Vietoris sequence for the triple $G_1,G_2,H$, one can then
deduce that the map $H_2(G_1,\R)\oplus H_2(G_2,\R)\to H_2(G,\R)$ is surjective, which implies in turn that the restriction map
$r\colon H^2(G,\R)\to H^2(G_1,\R)\oplus H^2(G_2,\R)$ is injective.

Let us now consider the commutative diagram
$$
\xymatrix{
 H^2_b(G,\R)\ar[rr]^(.37){r_b}\ar[d]^c & & H^2_b(G_1,\R)\oplus H^2_b(G_2,\R)\ar[d]^{c'} \\
H^2(G,\R)\ar[rr]^(.37){r} & & H^2(G_1,\R)\oplus H^2(G_2,\R)\ .
}
$$
Since $\alpha$ is weakly bounded, its restrictions in $H^2(G_1,\Z)$ and in $H^2(G_2,\Z)$ are also weakly bounded, hence bounded, since $G_1$ and $G_2$ satisfy QITB.
Together with the fact that the change of coefficients map takes bounded cochains to bounded cochains (and commutes with restrictions), 
this implies that $r(\alpha_\R)$ lies in the image of the comparison map $c'$. Since $H$ is amenable, the map $r_b$ is surjective (see~\cite{BBFIPP}), hence there exists
$\beta\in  H^2_b(G,\R)$ such that $c'(r_b(\beta))=r(\alpha_\R)$. Using that $r$ is injective we then get $c(\beta)=\alpha_\R$, i.e.~$\alpha_\R$ lies in the image of the comparison map.
This concludes the proof.
 \end{proof}
 
\begin{rem}
Let $G_1,G_2$, $H$ and $G$ be as in the statements of the previous propositions. The proofs above may be easily adapted to show the following:
\begin{enumerate}
\item If both $G_1$ and $G_2$ have Property $(*)$ (resp. $(**)$), then $G_1\times G_2$ has Property $(*)$.
\item If both $G_1$ and $G_2$ have Property $(*)$ (resp. $(**)$), then $G_1*_H G_2$ has Property $(*)$ (resp. $(**)$).
\end{enumerate}
\end{rem}

\subsection*{$3$-manifold groups}

In order to prove that $3$-manifold groups also satisfy Property QITB we first recall that $C_2(G,\R)$ is endowed with an $\ell^1$-norm such that $\|c\|_1=\sum |a_{(g_1,g_2)}|$ for every 
 chain $c=\sum a_{(g_1,g_2)} (g_1,g_2)$. We then endow 
 $H_2(G,\R)$ with the induced quotient $\ell^1$ seminorm  (which is sometimes called the \emph{Gromov seminorm}) such that, if $\beta\in H_2(G,\R)$, then
$\|\beta\|_1$ is the infimum of the $\ell^1$-norms of the representatives of $\beta$ in $C_2(G,\R)$.
Let us denote by $N_2(G,\R)$ the subspace of $H_2(G,\R)$ given by classes with vanishing $\ell^1$-seminorm. 
It is well known that, if $A$ is an amenable group, then the $\ell^1$-seminorm vanishes on $H_2(A,\R)$. Since group homomorphisms induce seminorm non-increasing
maps on homology, this readily implies that $\ha(G,\R)\subseteq N_2(G,\R)$.

\begin{prop}\label{criterio1}
 Let $G$ be a group such that $H_2(G,\R)$ is finite dimensional (this is the case, e.g., if $G$ is finitely presented). Then $G$ satisfies $(**)$ if and only if 
 $N_2(G,\R)=\ha(G,\R)$.
 \end{prop}
 \begin{proof}
 Suppose first that $N_2(G,\R)=\ha(G,\R)$, and take
 $\alpha\in \ann(\ha(G,\R))=\ann(N_2(G,\R))$. Then $\alpha$ defines a linear map $H_2(G,\R)/N_2(G,\R)\to\R$. Since $H_2(G,\R)$ is finite dimensional, this map is continuous
 with respect to the quotient $\ell^1$-norm on $H_2(G,\R)/N_2(G,\R)$.
 By~\cite[Proposition 1.1]{BarGhys}, this implies that $\alpha$ may be represented by a bounded cocycle. Thus $G$ satisfies $(**)$.
 
 Suppose now that $N_2(G,\R)\neq \ha(G,\R)$, and take an element $\beta\in N_2(G,\R)\setminus \ha(G,\R)$. By the Universal Coefficient Theorem,
 we may construct an element $\alpha\in \ann(\ha(G,\R))$ such that $\langle \alpha,\beta\rangle=1$. Since $\|\beta\|_1=0$, the class $\alpha$ cannot be represented by any bounded cocycle. Thus $G$ does not satisfy $(**)$.  
   \end{proof}

\begin{thm}\label{3manifolds}
 Let $G$ be the fundamental group of a compact orientable $3$-manifold. Then $G$ satisfies Property QITB.
 \end{thm}
\begin{proof}
Let $M$ be a compact orientable $3$-manifold. The decomposition of $M$ into prime summands $M_1,\dots,M_k$ decomposes $G$ as the free
product of the groups $G_i=\pi_1(M_i)$, $i=1,\dots,k$. By Corollary~\ref{free:cor}, we may thus assume that $M$ is prime. Moreover, if $M\cong S^2\times S^1$, then
$\pi_1(M)=\Z$ obviously satisfies Property QITB, hence we are reduced to study the case when $M$ is irreducible. We will show that, under this assumption, we have
that $G=\pi_1(M)$ satisfies (**), hence Property QITB.

Being the fundamental group of a compact manifold, $G$ is finitely presented, hence by Proposition~\ref{criterio1} it suffices to show that $N_2(G,\R)=\ha(G,\R)$. Recall that
the inclusion $\ha(G,\R)\subseteq N_2(G,\R)$ always holds, and take an element $\beta\in N_2(G,\R)$. If $G$ is finite, then of course $\beta \in \ha(G,\R)$. 
Since irreducible $3$-manifolds with infinite fundamental groups are aspherical, we may thus identify $H_2(G,\R)$ with $H_2(M,\R)$. 
The module $H_2(M,\R)$ is itself endowed with an $\ell^1$-seminorm (see e.g.~\cite{Gromov}), and the identification $H_2(G,\R)\cong H_2(M,\R)$ is isometric,
hence we may consider $\beta$ as an element of $H_2(M,\R)$ with vanishing seminorm. A result of Gabai~\cite[Corollary 6.18]{Gabai} now ensures
that, since $\|\beta\|_1=0$, also the Thurston norm of $\beta$ vanishes. Therefore, as an element of $H_2(M,\R)$, the class
$\beta$ is represented by a finite union of spheres and tori (in fact, since $M$ is aspherical, by a finite union of tori)~\cite{Thurston:norm}.
This immediately implies that $\beta\in \ha(G,\R)$, whence the conclusion.


\end{proof}

\bibliography{biblionote}
\bibliographystyle{alpha}

\end{document}